\documentclass[a4paper,10pt]{article}

\usepackage[applemac]{inputenc}   
\usepackage[T1]{fontenc}      
\usepackage{geometry}         
\usepackage[english]{babel}
\usepackage{amsthm}
\usepackage{amsmath}
\usepackage{amsfonts}
\usepackage{amssymb}
\usepackage[all]{xy}
\usepackage[pdftex]{hyperref}
\newtheorem{theoreme}{Theorem}[section]
\newtheorem{proposition}[theoreme]{Proposition}

\newtheorem{definition}[theoreme]{Definition}
\newtheorem{corollary}[theoreme]{Corollary}
\theoremstyle{definition} \newtheorem{remark}[theoreme]{Remark}
\theoremstyle{definition} \newtheorem{example}[theoreme]{Example}

\newcommand{\fldr}[1]{\stackrel{#1}\rightarrow}
\newcommand{\fldrs}[1]{\stackrel{#1}\twoheadrightarrow}
\newcommand{\fldri}[1]{\stackrel{#1}\hookrightarrow}
\newcommand{\lfldr}[1]{\stackrel{#1}\longrightarrow}

\newcommand{\dpart}[1]{\displaystyle{\left.\frac{\partial}{\partial{#1}}
\right|_{{#1} = 0}}}
\newcommand{\dpartt}[1]{\displaystyle{\frac{\partial}{\partial{#1}}}}
\newcommand{\somme}[2]{\displaystyle\sum_{#1}^{#2}}
\def\R {{\mathbb{R}}}
\def\N {{\mathbb{N}}}
\def\aa {{\mathfrak{a}}}
\def\gg {{\mathfrak{g}}}

\def\ggc{{\hat{\mathfrak{g}}}}

\def\xx {{\mathfrak{x}}}
\def\conj { {\rhd} }

\title{The local integration of Leibniz algebras}
\date{}
\author{Simon Covez}

\begin{document}

\maketitle

\begin{abstract}
This article gives a local answer to the coquecigrue problem. Hereby we mean the problem, formulated by J-L. Loday in \cite{LodayEns}, is that of finding a generalization of the Lie's third theorem for Leibniz algebra. That is, we search a manifold provided with an algebraic structure which generalizes the structure of a (local) Lie group, and such that the tangent space at a distinguished point is a Leibniz algebra structure. Moreover, when the Leibniz algebra is a Lie algebra, we want that the integrating manifold is a Lie group. In his article \cite{Kinyon}, M.K. Kinyon solves the particular case of split Leibniz algebras. He shows, in particular, that the tangent space at the neutral element of a Lie rack is provided with a Leibniz algebra structure. Hence it seemed reasonable to think that Lie racks give a solution to the coquecigrue problem, but M.K. Kinyon also showed that a Lie algebra can be integrated into a Lie rack which is not a Lie group. Therefore, we have to specify inside the category of Lie racks, which objects are the coquecigrues. In this article we give a local solution to this problem. We show that every Leibniz algebra becomes integrated into a \textit{local augmented Lie rack}. The proof is inspired by E. Cartan's proof of Lie's third theorem, and, viewing a Leibniz algebra as a central extension by some center, proceeds by integrating explicitely the corresponding Leibniz 2-cocycle into a rack 2-cocycle. This proof gives us a way to construct local augmented Lie racks which integrate Leibniz algebras, and this article ends with an example of the integration of a non split Leibniz algebra in dimension $5$.   
\end{abstract}
\section*{Introduction}
The main result of this article is a local answer to the \textit{coquecigrue problem}. By coquecigrue problem, we mean the problem of integrating Leibniz algebras. This question was formulated by J.-L. Loday in \cite{LodayEns} and consists in finding a generalisation of the Lie's third theorem for Leibniz algebras. This theorem establishes that for every Lie algebra $\gg$, there exists a Lie group $G$ such that its tangent space at $1$ is provided with a structure of Lie algebra isomorphic to $\gg$. Leibniz algebras are generalisations of Lie algebras, they are their non-commutative analogues. Precisely, a \textit{(left) Leibniz algebra} (over $\mathbb{R}$) is an $\mathbb{R}$-vector space $\gg$ provided with a bilinear map $[-,-]: \gg \times \gg \to \gg$ called the bracket and satisfying the \textit{(left) Leibniz identity} for all $x,y$ and $z$ in $\gg$
$$
[x,[y,z]] = [[x,y],z] + [y,[x,z]]
$$ 
Hence, a natural question is to know if, for every Leibniz algebra, there exists a manifold provided with an algebraic structure generalizing the group structure, and such that the tangent space at a distinguished point, called $1$, can be provided with a Leibniz algebra structure isomorphic to the given Leibniz algebra. As we want this integration to be the generalization of the Lie algebra case, we also require that, when the Leibniz algebra is a Lie algebra, the integrating manifold is a Lie group.\par
One result about this question was given by M.K. Kinyon in \cite{Kinyon}. In his article he solves the particular case of \textit{split} Leibniz algebras. A Leibniz algebras is split when it is isomorphic to the \textit{demisemidirect product} of a Lie algebra and a module over this Lie algebra, that is isomorphic to $\gg \oplus \aa$ as vector space and where the bracket is given by $[(x,a),(y,b)] = ([x,y],x.a)$. In this case he shows that the algebraic structure which answers the problem is the structure of a \textit{digroup}. A digroup is a set  with two binary operation $\vdash$ and $\dashv$, a neutral element $1$ and some compatibility conditions. More precisely, Kinyon shows that a digroup structure induces a \textit{pointed rack} structure (pointed in $1$), and it is this algebraic structure which gives the tangent space at $1$ a Leibniz algebra structure. Of course, not every Leibniz algebra is isomorphic to a demisemidirect product, so we have to find a more general structure to solve the problem. One should think that the right structure is that of a pointed rack, but M.K. Kinyon showed in \cite{Kinyon} that the second condition (Lie algebra becomes integrated into a Lie group) is not always fulfilled. Thus we have to specify the structure inside the category of pointed racks.\par
In this article we don't give a complete answer to the coquecigrue problem in the sense that we only construct a \textit{local} algebraic structure and not a global one. Indeed, to define an algebraic structure on a tangent space at a given point on a manifold, we just need an algebraic structure in a neighborhood of this point. We will show in chapter 3 that a local answer to the problem is given by the pointed augmented local racks which are abelian extensions of a Lie group by an anti-symmetric module.\par
Our approach to the problem is similar to the one given by E. Cartan in \cite{ECartan1}. The main idea comes from the fact that we know the Lie's first and second theorem on a class of Lie algebras. For example, every abelian Lie algebra or every Lie subalgebra of the Lie algebra $End(V)$ is integrable (using the Lie's first theorem). More precisely, let $\gg$ be a Lie algebra, $Z(\gg)$ its center and $\gg_0$ the quotient of $\gg$ by $Z(\gg)$. The Lie algebra $Z(\gg)$ is abelian and $\gg_0$ is a Lie subalgebra of $End(\gg)$, thus there exist Lie groups, respectively $Z(\gg)$ and $G_0$, which integrate these Lie algebras. As a vector space, $\gg$ is isomorphic to the direct sum $\gg_0 \oplus Z(\gg)$,  thus the tangent space at $(1,0)$ of the manifold $G_0 \times Z(\gg)$ is isomorphic to $\gg$. As a Lie algebra, $\gg$ is isomorphic to the central extension $\gg_0 \oplus_{\omega} Z(\gg)$ where $\omega$ is a Lie $2$-cocycle on $\gg_0$ with coefficients in $Z(\gg)$. That is, the bracket on $\gg_0 \oplus_{\omega} Z(\gg)$ is defined by
\begin{eqnarray}
[(x,a),(y,b)] = \big([x,y],\omega(x,y)\big) \label{crochet dans g  Z(g)}
\end{eqnarray}
where $\omega$ is an anti-symmetric bilinear form on $\gg_0$ with value on $Z(\gg)$ which satisfies the Lie algebra cocycle identity
$$
\omega([x,y],z) - \omega(x,[y,z]) + \omega(y,[x,z]) = 0
$$           
Hence we have to find a group structure on $G_0 \times Z(\gg)$ which gives this Lie algebra structure on the tangent space at $(1,0)$. It is clear that the bracket \eqref{crochet dans g  Z(g)} is completely determined by the bracket on $\gg_0$ and the cocycle $\omega$. Hence, the only thing we have to understand is $\omega$. The Lie algebra $\gg$ is a central extension of $\gg_0$ by $Z(\gg)$, thus we can hope that the Lie group which integrates $\gg$ should be a central extension of $G_0$ by $Z(\gg)$. To follow this idea, we have to find a group $2$-cocycle on $G_0$ with coefficients in $Z(\gg)$. In this case, the group structure on $G_0 \times Z(\gg)$ is given by
\begin{eqnarray}
(g,a).(h,b) = \big(gh, a + b + f(g,h)\big) \label{produit dans G x Z(g)}
\end{eqnarray}
where $f$ is a map from $G \times G \to Z(\gg)$ vanishing on $(1,g)$ and $(g,1)$ and satisfying the group cocycle identity
$$
f(h,k) - f(gh,k) + f(g,hk) - f(g,h) = 0
$$   
With such a cocycle, the conjugation in the group is given by the formula
\begin{eqnarray}
(g,a).(h,b).(g,a)^{-1} = \big(ghg^{-1}, a + f(g,h) - f(ghg^{-1},g)\big) \label{conjugaison dans G x Z(g)}
\end{eqnarray}
and by imposing a smoothness condition on $f$ in a neighborhood of $1$, we can differentiate this formula twice, and obtain a bracket on $\gg_0 \oplus Z(\gg)$ defined by
$$
[(x,a),(y,b)] = \big([x,y], D^2f(x,y)\big)
$$ 
where $D^2f(x,y) = d^2f(1,1)((x,0),(0,y)) - d^2f(1,1)((y,0),(0,x))$. Thus, if $D^2f(x,y)$ equals $\omega(x,y)$, we recover the bracket \eqref{crochet dans g  Z(g)}. Hence, if we associate to $\omega$ a group cocyle $f$ satisfying some smoothness conditions and such that $D^2f = \omega$, then our integration problem is solved. This can be done in two steps. The first one consists in finding a local Lie group cocycle defined around $1$. Precisely, we want a map $f$ defined on a subset of $G_0 \times G_0$ containing $(1,1)$ with values in $Z(\gg)$ which satisfies the local group cocycle identity (cf. \cite{VanEst} for a definition of local group). We can construct explicitely such a local group cocycle. This construction is the following one (cf. Lemma 5.2 in \cite{Neeb2}) : \par
Let $V$ be an open convex 0-neighborhood in $\gg_0$ and $\phi:V \to G_0$ a chart of $G_0$ with $\phi(0)=1$ and $d\phi(0) = id_{\gg_0}$. For all $(g,h) \in \phi(V) \times \phi(V)$ such that $gh \in \phi(V)$ let us define $f(g,h) \in Z(\gg)$ by the formula
$$
f(g,h) = \int_{\gamma_{g,h}} \omega^{inv}
$$
where $\omega^{inv} \in \Omega^2(G_0,Z(\gg))$ is the invariant differential form on $G_0$ associated to $\omega$ and $\gamma_{g,h}$ is the smooth singular $2$-chain defined by
$$
\gamma_{g,h}(t,s) = \phi\Bigg(t\bigg(\phi^{-1}\Big(g\phi \big(s\phi^{-1}(h)\big)\Big)\bigg) + s\bigg(\phi^{-1}\Big(g\phi \big((1-t)\phi^{-1}(h)\big)\Big)\bigg)\Bigg)
$$
The formula  for $f$ defines a smooth function such that $D^2f(x,y) =\omega(x,y)$. We now only have to check whether $f$ satisfies the local group cocycle identity. Let $(g,h,k) \in \phi(V)^3$ such that $gh,hk$ and $ghk$ are in $\phi(V)$. We have
\begin{align*}
f(h,k) - f(gh,k) + f(g,hk) - f(g,h) &=  \int_{\gamma_{h,k}} \omega^{inv} -  \int_{\gamma_{gh,k}} \omega^{inv} +  \int_{\gamma_{g,hk}} \omega^{inv} -  \int_{\gamma_{g,h}} \omega^{inv} \\
						&= \int_{\partial\gamma_{g,h,k}} \omega^{inv}
\end{align*}
where $\gamma_{g,h,k}$ is a smooth singular $3$-chain in $\phi(V)$ such that $\partial\gamma_{g,h,k} = g\gamma_{h,k} - \gamma_{gh,k} + \gamma_{g,hk} - \gamma_{g,h}$ (such a chain exists because $\phi(V)$ is homeomorphic to the convex open subset $V$ of $\gg_0$). Thus
\begin{align*}
f(h,k) - f(gh,k) + f(g,hk) - f(g,h) &= \int_{\partial\gamma_{g,h,k}} \omega^{inv} = \int_{\gamma_{g,h,k}} d_{dR}\omega^{inv} = 0
\end{align*}
because $\omega^{inv}$ is a closed $2$-form. Hence, we have associated to $\omega$ a local group $2$-cocycle, smooth in a neighborhood in $1$, and such that $D^2f(x,y) = \omega(x,y)$. Thus we can define a local Lie group structure on $G_0 \times Z(\gg)$ by setting
$$
(g,a)(h,b) = \big(gh, a + g.b + f(g,h)\big),
$$
and the tangent space at $(1,0)$ of this local Lie group is isomorphic to $\gg$. If we want a global structure, we have to extend this local cocycle to the whole group $G_0$. First P.A. Smith (\cite{Smith1,Smith2}), then W.T. Van Est (\cite{VanEst2,VanEst3}) have shown that it is precisely this enlargement which may meet an obstruction coming from both $\pi_2(G_0)$ and $\pi_1(G_0)$. In finite dimension, $\pi_2(G_0) =  0$, thus there is no obstruction to integrate Lie algebras. This equality is no longer true in infinite dimension, hence this obstruction prevents the integration of infinite dimensional Lie algebras into global Lie groups (cf. \cite{Neeb1,Neeb2}).\par
To integrate Leibniz algebras into pointed racks, we follow a similar approach. In this context, we use the fact that we know how to integrate any (finite dimensional) Lie subalgebra of $End(V)$ for $V$ a vector space. In a similar way as the Lie algebra case, we associate to any Leibniz algebra an abelian extension of a Lie algebra $\gg_0$ by an anti-symmetric representation $Z_L(\gg)$. As we have the theorem for Lie algebras, we can integrate $\gg_0$ and $Z_L(\gg)$ into the Lie groups $G_0$ and $Z_L(\gg)$, and, using the Lie's second theorem, $Z_L(\gg)$ is a $G_0$-module. Then, the main difficulty becomes the integration of the Leibniz cocycle into a local Lie rack cocycle. In chapter $3$ we explain how to solve this problem. We make a similar construction as in the Lie algebra case, but in this context, there are several difficulties which appear. One of them is that our cocycle is not anti-symmetric, so we can't consider the equivariant form associated to it and integrate this form. To solve this problem, we will use Proposition \ref{isomorphisme entre CLn et CLn-1} which, in particular, establishes an isomorphism from the $2$-nd cohomology group of a Leibniz algebra $\gg$ with coefficients in an anti-symmetric representation $\aa^a$ to the $1$-st cohomology group of $\gg$ with coefficients in the symmetric representation $Hom(\gg,\aa)^s$. In this way, we get a $1$-form that we can now integrate. Another difficulty is to specify on which domain this $1$-form should be integrated. In the Lie algebra case, we integrate over a $2$-simplex and the cocycle identity is verified by integrating over a $3$-simplex, whereas in our context we will replace the $2$-simplex by the $2$-cube and the $3$-simplex by a $3$-cube.
\vskip 0.2cm
\noindent
Let us describe the content of the article section-wise.
\subsection*{Section 1: Leibniz algebras}
This whole section, except Proposition \ref{isomorphisme entre CLn et CLn-1}, is based on \cite{LodayEns,Loday,LodayCyclic}.  We first give the basic definitions we need about Leibniz algebras. Unlike J.-L. Loday and T. Pirashvili, who work with right Leibniz algebras, we study left Leibniz algebras. Hence, we have to translate all the definitions needed into our context. As we have seen above, we translate our integration problem into a cohomological problem, thus we need a cohomology theory for Leibniz algebras and, a fortiori, a notion of representation. We take the definition of a representation of a Leibniz algebra given by J.-L. Loday and T. Pirashvili in \cite{Loday}. We end this section with a fondamental result (Proposition \ref{isomorphisme entre CLn et CLn-1}). This proposition establishes an isomorphism of cochain complexes from $CL^n(\gg,\aa^a)$ to $CL^{n-1}(\gg,Hom(\gg,\aa)^s)$. The important fact in this result is the transfer from an anti-symmetric representation to a symmetric one. This will be useful when we will have to associate a local Lie rack $2$-cocycle to a Leibniz $2$-cocycle.
\subsection*{Section 2: Lie racks}
The notion of rack comes from topology, in particular, the theory of invariants of knots and links (cf. for example \cite{FennRourke}). It is M.K. Kinyon in \cite{Kinyon} who was the first to  link racks to Leibniz algebras. The idea of linking these two structures comes from the case of Lie groups and Lie algebras and in particular from the construction of the bracket using the conjugation. Indeed, a way to define a bracket on the tangent space at $1$ of a Lie group is to differentiate the conjugation morphism twice. Let $G$ a Lie group, the conjugation is the group morphism $c: G \to Aut(G)$ defined by $c_g(h) = ghg^{-1}$. If we differentiate this expression with respect to the variable $h$ at $1$, we obtain a Lie group morphism $Ad:G \to Aut(\gg)$. We can still derive this morphism at $1$ to obtain a linear map $ad: \gg \to End(\gg)$. Then, we are allowed to define a bracket $[-,-]$ on $\gg$ by setting $[x,y] = ad(x)(y)$. We can show that this bracket satisfies the left Leibniz identity, and that this identity is induced by the equality $c_g(c_h(k)) = c_{c_g(h)}(c_g(k))$. Thus, if we denote $c_g(h)$ by $g \conj h$, the only properties we use to define a Lie bracket on $\gg$ are
\begin{enumerate}
\item $g \conj: G \to G$ is a bijection for all $g \in G$.
\item $g \conj (h \conj k) = (g \conj h) \conj (g \conj k)$ for all $g,h,k \in G$
\item $g \conj 1= 1$ and $1 \conj g = g$ for all $g \in G$. 
\end{enumerate}
Hence, we call \textit{(left) rack}, a set provided with a binary operation $\conj$ satisfying the first and the second condition. A rack is called \textit{pointed} if there exists an element $1$ which satisfies the third condition. We begin this chapter by giving definitions and examples, for this we follow \cite{FennRourke}. They work with right racks, hence, as in the Leibniz algebra case, we translate the definitions to left racks. In particular, we give the most important example called \textit{(pointed) augmented rack}. 
This example presents similarities with crossed modules of groups, and in this case, the rack structure is induced by a group action.\par
As in the group case, we want to construct a pointed rack associated to a Leibniz algebra using an abelian extension. Hence, we need a cohomology theory where the second cohomology group corresponds to the extension classes of a rack by a module. We take the definitions given by N. Andruskiewitsch and M. Graña in \cite{AndruskiewitschGrana}.
\par
At the end of this section, we give the definitions of local rack cohomology and (local) Lie rack cohomology.    
\subsection*{Section 3: Lie racks and Leibniz algebras}
This section is the heart of this article. It gives the local solution for the coquecigrue problem. To our knowledge, all the results in this chapter are new, except Proposition \ref{l'espace tangent d'un rack est Leibniz} due to M.K. Kinyon (\cite{Kinyon}). First, we recall the link between (local) Lie racks and Leibniz algebras explained by M.K. Kinyon in \cite{Kinyon} (Proposition \ref{l'espace tangent d'un rack est Leibniz}). Then, we study the passage from smooth $As(X)$-modules to Leibniz representations (Proposition \ref{Lien entre smooth As(X)-modules et Leibniz representations}) and (local) Lie rack cohomology to Leibniz cohomology. We define a morphism from the (local) Lie rack cohomology of a rack $X$ with coefficients in a $As(X)$-module $A^s$ (resp. $A^a$) to the Leibniz cohomology of the Leibniz algebra associated to $X$ with coefficients in $\aa^s = T_0A$ (resp. $\aa^a$) (Proposition \ref{morphisme de la cohomologie de rack vers la cohomologie de Leibniz}).
\par 
The end of this section (section $3.4$ to $3.7$) is on the integration of Leibniz algebras into local Lie racks. We use the same approach as E. Cartan for the Lie groups case. That is, for every Leibniz algebra, we consider the abelian extension by the left center and integrate it. This extension is caracterized by a $2$-cocycle, and we construct (Proposition \ref{I2 inverse à gauche de D2}) a local Lie rack $2$-cocycle integrating it by an explicit construction similar to the one explained in the Lie group case. This construction is summarized in our main theorem (Theorem \ref{theoreme principal}). We remark that the constructed $2$-cocycle has more structure (Proposition \ref{identité fantôme de I^2}). That is, the rack cocycle identity is induced by another one. This other identity permits us to provide our constructed local Lie rack with a structure of  augmented local Lie rack (Proposition \ref{G_0 x a$ est un rack augmenté}). We end this section with an example of the integration of a non split Leibniz algebra in dimension $5$.     
\section{Leibniz algebras}
As it is written in the introduction, we work with left Leibniz algebras instead of right Leibniz algebras. The main reason comes from the fact that M.K. Kinyon works in this context in his article \cite{Kinyon}. Indeed, this article is our starting point of the integration problem for Leibniz algebras. Thus, we have chosen to work in this context.\par
A \textit{(left) Leibniz algebra (over $\R$)} is a vector space $\gg$ (over $\R$) provided with a bracket $[-,-]: \gg \otimes \gg \to \gg$, which satisfies the \textit{left Leibniz identity}
$$
[x,[y,z]] = [[x,y],z] + [y,[x,z]].
$$  
Remark that an equivalent way to define a left Leibniz algebra is to say that, for all $x \in \gg$, $[x,-]$ is a derivation for the bracket $[-,-]$. The first example of a Leibniz algebra is a Lie algebra. Indeed, if the bracket is anti-symmetric, then the Leibniz identity is equivalent to the Jacobi identity. Hence, we have a functor $inc: Lie \to Leib$. This functor has a left adjoint $(-)_{Lie}:Leib \to Lie$ which is defined on the objects by $\gg_{Lie} = \gg / \gg_{ann}$, where $\gg_{ann}$ is the two-sided ideal of $\gg$ generated by the set $\{[x,x] \in \gg \, | \, x \in \gg\}$. We can remark that there are other ways to construct a Lie algebra from a Leibniz algebra. One is to quotient $\gg$ by the \textit{left center} $Z_L(\gg) = \{x \in \gg \, | \, [x,-] = 0\}$, but this construction is not functorial.
\vskip 0.2cm
To define a cohomology theory for Leibniz algebras, we need a notion of representation of such algebraic structure. As we work with left Leibniz algebra, we have to translate the definition given by J.L. Loday and T. Pirashvili in their article \cite{Loday}. In our context, a \textit{representation} over a Leibniz algebra $\gg$, becomes a vector space $M$ provided with two linear maps $[-,-]_L:\gg \otimes M \to M$ and $[-,-]_R:M \otimes \gg \to \gg$, satisfying the following three axioms
\begin{align*}
  [x,[y,m]_{L}]_{L} &= [[x,y],m]_{L} + [y,[x,m]_{L}]_{L} \ \ \tag{$LLM$}\\
  \lbrack x,[m,y]_{R}]_{L} &= [[x,m]_{L},y]_{R} + [m,[x,y]]_{R} \ \ \tag{$LML$}\\
  \lbrack m,[x,y]]_{R} &= [[m,x]_{R},y]_{R} + [x,[m,y]_{R}]_{L} \ \ \tag{$MLL$}
\end{align*}  
Recall that, for a Lie algebra $\gg$, a representation of $\gg$ is a vector space $M$ provided with a linear map $[-,-]:\gg \otimes M \to M$ satisfying $[[x,y],m] = [x,[y,m]] - [y,[x,m]]$. A Lie algebra is a Leibniz algebra, hence we want that a Lie representation $M$ of a Lie algebra $\gg$, is a Leibniz representation of $\gg$. We have two canonical choices for putting a Leibniz representation structure on $M$. One possibility is by setting $[-,-]_L = [-,-]$ and $[-,-]_R = -[-,-]$, and a second one is by setting $[-,-]_L = [-,-]$ and $[-,-]_R = 0$. These Leibniz representations are examples of particular Leibniz representations. The first one is an example of a \textit{symmetric} representation, and the second one is an example of an \textit{anti-symmetric} representation. A symmetric representation is a Leibniz representation where $[-,-]_L = -[-,-]_R$ and an anti-symmetric representation is a Leibniz representation where $[-,-]_R = 0$. A Leibniz representation which is symmetric and anti-symmetric is called \textit{trivial}.
\vskip 0.2cm
Now, we are ready to define a cohomology theory for Leibniz algebras. The existence of a cohomology (and homology) theory for these algebras is one of the main motivation for studying them because, restricted to Lie algebras, this theory gives us new invariants (cf. \cite{LodayEns}). For $\gg$ a Leibniz algebra and $M$ a representation of $\gg$, we define a cochain complex $\{CL^n(\gg,M),dL^n\}_{n \in \N}$ by setting
$$
CL^n(\gg,M) = Hom(\gg^{\otimes n},M)
$$   
and 
$$
\begin{array}{lll}
dL^{n}\omega(x_{0},\dots,x_{n}) & = \ \ \displaystyle{\sum_{i=0}^{n-1}} (-1)^{i}[x_{i},\omega(x_{0},\dots,\hat{x_{i}},\dots,x_{n})]_{L} +  (-1)^{n-1}[\omega(x_{0},\dots,x_{n-1}),x_{n}]_{R} \\ & \ \ \ \ + \displaystyle{\sum_{0\leq i < j \leq n}} (-1)^{i+1} \omega(x_{0},\dots,x_{j-1},[x_{i},x_{j}],x_{j+1},\dots,x_{n})
\end{array}
$$
To prove that $dL^{n+1} \circ dL^n = 0$, we use \textit{Cartan's formulas}. These formulas are described in \cite{Loday} in the right Leibniz algebra context, but we can adapt them easily in our context.\par
Like for many algebraic structures, the second cohomology group of a Leibniz algebra $\gg$ with coefficients in a representation $M$ is in bijection with the set of equivalence classes of abelian extensions of $\gg$ by $M$ (cf. \cite{Loday}). \textit{An abelian extension} of a Leibniz algebra $\gg$ by $M$ is a Leibniz algebra $\ggc$ such that, $M \fldri{i} \ggc \fldrs{p} \gg$ is a short exact sequence of Leibniz algebra (where $M$ is considered as an abelian Leibniz algebra) and the representation structure of $M$ is compatible with the representation structure induced by this short exact sequence. That is, $[m,x]_R = i^{-1}[i(m),s(x)]$ and $[x,m]_L = i^{-1}[s(x),i(m)]$ where $s$ is a section of $p$ and the bracket is that of $\ggc$ (of course, we have to justify that this representation structure of $\gg$ on $M$ induced by the short exact sequence doesn't depend on $s$, but we deduce it easily from the fact that the difference of two sections of $p$ is in $i(M)$).
\par
There are canonical abelian extensions associated to a Leibniz algebra. The one we will use to integrate Leibniz algebras is the \textit{abelian extension by the left center}
$$
Z_L(\gg) \fldri{i} \gg \fldrs{p} \gg_0
$$
where $\gg_0 := \gg / Z_L(\gg)$. This is an extension of a Lie algebra by an anti-symmetric representation. In a sense, a symmetric representation is closer to a Lie representation than to an anti-symmetric representation. Hence, it is convenient to pass from a anti-symmetric representation to a symmetric representation. Let $\gg$ be a Lie algebra and $M$ a Lie representation of $\gg$, then we define a Lie representation structure on $Hom(\gg,M)$ by setting
$$
(x.\alpha)(y) := x.(\alpha(y)) - \alpha([x,y])
$$
for all $x,y \in \gg$ and $\alpha \in Hom(\gg,M)$. The following proposition establishes an isomorphism from $HL^n(\gg,M^a)$ to $HL^{n-1}(\gg,Hom(\gg,M)^s)$, where $M^a$ (resp. $Hom(\gg,M)^s$) means that $M$ (resp. $Hom(\gg,M)$) is provided with a anti-symmetric (resp. symmetric) $\gg$-representation structure.
\begin{proposition} \label{isomorphisme entre CLn et CLn-1}
Let $\gg$ be a Lie algebra and $M$ a Lie representation of $\gg$. We have an isomorphism of cochain complexes 
$$
CL^{n}(\gg,M^{a}) \fldr{\tau^{n}} CL^{n-1}(\gg,Hom(\gg,M)^{s})
$$
given by $\omega \mapsto \tau^{n}(\omega)$ where $\tau^{n}(\omega)(x_{1},\dots,x_{n-1})(x_{n}) = \omega(x_{1},\dots,x_{n})$.
\end{proposition}
\begin{proof} This morphism is clearly an isomorphism $\forall n \geq 0$. Moreover, we have
\begin{align*} \displaystyle
dL\tau^{n}(\omega)(x_{0},\dots,x_{n-1})(x_{n}) &= \sum_{i=0}^{n-2} (-1)^{i}[x_{i},\tau^{n}(\omega)(x_{0},\dots,\widehat{x_{i}},\dots,x_{n-1})](x_{n}) \\ 
								& +  (-1)^{n-1}[x_{n-1},\tau^{n}(\omega)(x_{0},\dots,x_{n-2})](x_{n})\\ 
								& + \displaystyle{\sum_{0\leq i < j \leq n-1}} (-1)^{i+1} \tau^{n}(\omega)(x_{0},\dots,x_{j-1},[x_{i},x_{j}],x_{j+1},\dots,x_{n-1})(x_{n})\\
							         & = \sum_{i=0}^{n-1} (-1)^{i}([x_{i},\omega(x_{0},\dots,\hat{x_{i}},\dots,x_{n-1},x_{n})] - \omega(x_{0},\dots,\widehat{x_{i}},\dots,x_{n-1},[x_{i},x_{n}]))\\
							         & + \displaystyle{\sum_{0\leq i < j \leq n-1}} (-1)^{i+1} \omega(x_{0},\dots,x_{j-1},[x_{i},x_{j}],x_{j+1},\dots,x_{n-1},x_{n})
							         \end{align*}
\begin{align*}
dL\tau^{n}(\omega)(x_{0},\dots,x_{n-1})(x_{n})	 & = \sum_{i=0}^{n-1} (-1)^{i}[x_{i},\omega(x_{0},\dots,\widehat{x_{i}},\dots,x_{n-1},x_{n})]\\
							         & + \sum_{0\leq i < j \leq n} (-1)^{i+1} \omega(x_{0},\dots,x_{j-1},[x_{i},x_{j}],x_{j+1},\dots,x_{n-1},x_{n}) \\
							         & = dL\omega(x_{0},\dots,x_{n-1},x_{n})\\
							         & = \tau^{n+1}(dL\omega)(x_{0},\dots,x_{n-1})(x_{n})
\end{align*}
Hence $\{\tau^{n}\}_{n \geq 0}$ is a morphism of cochain complexes.
\end{proof}
\section{Lie racks}
\subsection{Definitions and examples}
Like in the Leibniz algebra case, we can define left racks and right racks. Because we have made the choice to work with left Leibniz algebras, we take the definition of left racks. A \textit{(left) rack} is a set $X$ provided with a product $\conj: X \times X \to X$, which satisfies the \textit{left rack identity}, that is for all $x,y,z \in X$ :
$$
x \conj (y \conj z) = (x \conj y) \conj (x \conj z),
$$
and such that $x \conj_{-} : X \to X$ is a bijection for all $x \in X$. A rack is said to be \textit{pointed} if there exists an element $1 \in X$, called the neutral element, which satisfies $1 \conj x = x$ and $x \conj 1 = 1$ for all $x \in X$. A \textit{rack morphism} is a map $f:X \to Y$ satisfying $f(x \conj y) = f(x) \conj f(y)$, and a \textit{pointed rack morphism} is a rack morphism $f$ such that $f(1)=1$. \par
The first example of a rack is a group provided with the conjugation. Indeed, let $G$ be a group, we define a rack product $\conj$ on $G$ by setting $g \conj h = ghg^{-1}$ for all $g,h \in G$. Clearly, $g \conj_{-}$ is a bijection with inverse $g^{-1} \conj_{-}$ and, an easy computation shows that the rack identity is satisfied. Hence, we have a functor $Conj:Group \to Rack$. This functor has a left adjoint $As:Rack \to Group$ defined on the objects by $As(X) = F(X) / <\{xyx^{-1}(x\conj y^{-1}) \, | \, x,y \in X\}>$ where $F(X)$ is the free group generated by $X$, and $<\{xyx^{-1}(x\conj y^{-1}) \, | \, x,y \in X\}>$ is the normal subgroup generated by $\{xyx^{-1}(x\conj y^{-1}) \, | \, x,y \in X\}$. We can remark that $Conj(G)$ is a pointed rack. Indeed, we have $1 \conj g = g$ and $g \conj 1 = 1$ for all $g \in G$, where $1$ is the neutral element for the group product. Hence, $Conj$ is a functor from $Group$ to $PointedRack$. This functor has a left adjoint $As_p: PointedRack \to Group$, defined on the objects by $As_p(X) = As(X) / <\{[1]\}>$, where $<\{[1]\}>$ is the subgroup of $As(X)$ generated by the class $[1] \in As(X)$.
\par
A second example, and maybe the most important, is the example of augmented racks. An \textit{augmented rack} is the data of a group $G$, a $G$-set $X$, and a map $X \fldr{p} G$ satisfying the \textit{augmentation identity}, that is for all $g \in G$ and $x \in X$
$$
p(g.x) = gp(x)g^{-1}.
$$
Then, we define a rack structure on $X$ by setting $x \conj y = p(x).y$. If there exists an element $1 \in X$ such that $p(1) = 1$ and $g.1 = 1$ for all $g \in G$, then the augmented rack $X \fldr{p} G$ is said to be \textit{pointed}, and the associated rack $(X,\conj)$ is pointed. We can remark that crossed modules and precrossed modules of groups are examples of augmented racks.
\subsection{Pointed rack cohomology}
To define a pointed rack cohomology theory, we need a good notion of pointed rack module. In this article, we take the definition given by N. Andruskiewitsch and M. Graña in \cite{AndruskiewitschGrana}.  Let $X$ be a pointed rack, an \textit{X-module} is an abelian group $A$, provided with two families of homomorphisms of the abelian group $A$, $(\phi_{x,y})_{x,y \in X}$ and $(\psi_{x,y})_{x,y \in X}$, satisfying the following axioms
\begin{enumerate}
\item[$(M_{0})$] $\phi_{x,y}$ is an isomorphism.
\item[$(M_{1})$] $\phi_{x,y \conj z} \circ \phi_{y,z} = \phi_{x \conj y, x \conj z} \circ \phi_{x,z}$
\item[$(M_{2})$] $\phi_{x,y \conj z} \circ \psi_{y,z} = \psi_{x \conj y, x \conj z} \circ \phi_{x,y}$
\item[$(M_{3})$] $\psi_{x,y \conj z} = \phi_{x \conj y, x \conj z} \circ \psi_{x,z} + \psi_{x \conj y, x \conj z} \circ \psi_{x,y}$ 
\item[$(M_{4})$] $\phi_{1,y} = id_{A} \ \ \forall y \in X$ and $\psi_{x,1} = 0 \ \ \forall x \in X$
\end{enumerate}
\begin{remark}
There is a more general definition of (pointed) rack module given by N. Jackson in \cite{Jackson2}, but we don't need this degree of generality. This definition of pointed rack module coincides with the definition of homogeneous pointed rack module given in \cite{Jackson2}. 
\end{remark}
For example, there are two canonical $X$-module structures on an $As_p(X)$-module. Indeed, let $A$ be a $As_p(X)$-module, that is an abelian group provided with a group morphism $\rho:As_p(X) \to Aut(A)$, the first $X$-module structure, called \textit{symmetric}, that we can define on $A$ is given for all $x,y \in X$ by
\begin{align*}
\phi_{x,y}(a) &= \rho_x(a) \, \text{ and } \, \psi_{x,y}(a) = a - \rho_{x \conj y}(a).
\end{align*}
The second, called \textit{anti-symmetric}, is given for all $x,y \in X$ by
\begin{align*}
\phi_{x,y}(a) &= \rho_x(a) \, \text{ and } \, \psi_{x,y}(a) = 0.
\end{align*}
\par
With this definition of module, N. Andruskiewitsch and M. Graña define a cohomology theory for pointed racks. For $X$ a pointed rack and $A$ a $X$-module, they define a cochain complex $\{CR^n(X,A),d_R^n\}_{n \in \N}$ by setting
$$
CR^n(X,A) = \{f:X^n \to A \, | \, f(x_1,\dots,1,\dots,x_n) = 0\}
$$
and
$$
\begin{array}{ll}
d_R^nf(x_1,\dots,x_{n+1}) = \\ \somme{i = 1}{n} (-1)^{i-1}\big(\phi_{x_1 \conj \dots \conj x_i,x_1 \conj \dots \conj \widehat{x_i} \conj \dots \conj x_{n+1}}(f(x_1,\dots,\widehat{x_i},\dots,x_{n+1})) - f(x_1,\dots,x_i \conj x_{i+1},\dots,x_i \conj x_{n+1})\big) \\ \qquad + (-1)^{n}\psi_{x_1 \conj  \dots \conj x_n,x_1 \conj  \dots \conj x_{n-1} \conj x_{n+1}}(f(x_1,\dots,x_n))
\end{array}
$$
This complex is the same as the one defined in \cite{Jackson2}, but in the left rack context. Adapting the proof given by N. Jackson in \cite{Jackson2}, one easily sees that the second cohomology group $HR^2(X,A)$ is in bijection with the set of equivalence classes of abelian extensions of a pointed rack $X$ by a $X$-module $A$. An \textit{abelian extension} of a pointed rack $X$ by a $X$-module $A$ is a surjective pointed rack homomorphism $E \fldrs{p} X$ which satisfies the following axioms
\begin{enumerate}
\item[$(E_0)$] for all $x \in X$, there is a simply transitively right action of $A$ on $p^{-1}(x)$.
\item[$(E_1)$] for all $u \in p^{-1}(x), v \in p^{-1}(y), a \in A$, we have $(u.a) \conj v = (u \conj v).\psi_{x,y}(a)$.
\item[$(E_2)$] for all $u \in p^{-1}(x), v \in p^{-1}(y), a \in A$, we have $u \conj (v.a) = (u \conj v).\phi_{x,y}(a)$. 
\end{enumerate}
and two extensions $E_1 \fldrs{p_1} X$ , $E_2 \fldrs{p_2} X$ are called \textit{equivalent}, if there exists a pointed rack isomorphism $E_1 \fldr{\theta} E_2$ which satisfies the following axioms
\begin{enumerate}
\item $p_2 \circ \theta = p_1$.
\item for all $x \in X, u \in p^{-1}(x), a \in A$, we have $\theta(u.a) = \theta(u).a$.
\end{enumerate}
\subsection{Lie racks}
To generalize Lie groups, we need a pointed rack provided with a differentiable structure compatible with the algebraic structure. This is the notion of Lie racks. A \textit{Lie rack} is a smooth manifold $X$ provided with a pointed rack structure such that the product $\conj$ is smooth, and such that for all $x \in X \ \ c_x$ is a diffeomorphism. We will see in section $3$ that the tangent space at the neutral element of a Lie rack is provided with a Leibniz algebra structure.
\vskip0.2cm
Let $X$ be a Lie rack, a $X$-module $A$ is said \textit{smooth} if $A$ is a abelian Lie group, and if $\phi:X \times X \times A \to A$ and $\psi:X \times X \times A \to A$ are smooth. Then we can define a cohomology theory for Lie racks with values in a smooth module. For this we define a cochain complex $\{CR_s^n(X,A),d_R^n\}_{n \in \N}$ where $CR_s^n(X,A)$ is the set of functions $f:X^n \to A$ which are smooth in a neighborhood of $(1,\dots,1) \in X^n$ and such that $f(x_1,\dots,1,\dots,x_n) = 0$ for all $x_1,\dots,x_n \in X$. The formula for the differential $d_R$ is the same as the one defined previously. We will see that a Lie rack cocycle (respectively a coboundary) derives itself in a Leibniz algebra cocycle (respectively coboundary).
\subsection{Local racks}
To define a Lie algebra structure on the tangent space at the neutral element of a Lie group, we can remark that we only use the local Lie group structure in the neighborhood of $1$. We will see that this remark remains true for Lie racks and Leibniz algebras.
\vskip0.2cm
A \textit{local rack} is a set $X$ provided with a product $\conj$ defined on a subset $\Omega$ of $X \times X$ with values in $X$, and such that the following axioms are satisfied:
\begin{enumerate}
\item If $(x,y),(x,z),(y,z),(x,y \conj z)$ and $(x \conj y,x \conj z) \in \Omega$, then $x \conj (y \conj z) = (x \conj y) \conj (x \conj z)$.
\item If $(x,y),(x,z) \in \Omega$ and $x \conj y = x \conj z$, then $y = z$.
\end{enumerate}
A local rack is said to be \textit{pointed} if there is a element $1 \in X$ such that $1 \conj x$ and $x \conj 1$ are defined for all $x \in X$ and respectively equal to $x$ and $1$. We called this element the \textit{neutral element}. Then a \textit{local Lie rack} is a pointed local rack $(X,\Omega,1)$ where $X$ is a smooth manifold, $\Omega$ is a open subset of $X$, and $\conj: \Omega \to X$ is smooth. For example, every Lie rack open subset containing the neutral element is a local Lie rack. Given such a local Lie rack, we can define a associated cohomology theory.
\par
Let $X$ be a Lie rack, $U$ a subset of $X$ containing the neutral element $1$ and $A$ a smooth $X$-module. We define for all $n \in \N, \, CR_s^n(U,A)$ as the set of maps $f:U_{n-loc} \to A$, smooth in a neighborhood of the neutral element, and such that $f(x_1,\dots,1,\dots,x_n) = 0$. If $A$ is not anti-symmetric, then $U_{n-loc}$ is the subset of elements $(x_1,\dots,x_n)$ of $X \times U^{n-1}$ satisfying $x_{i_1} \conj \dots \conj x_{i_j} \in U$, for all $i_1 < \dots < i_j, 2 \leq j \leq n$. If $A$ is anti-symmetric, $U_{n-loc}$ is the subset of elements $(x_1,\dots,x_n)$ of $X^{n-1} \times U$ satisfying $x_{i_1} \conj \dots \conj x_{i_j} \conj x_n \in U$, for all $i_1 < \dots < i_j < n, 1 \leq j \leq n-1$. One easily checks that the formula for the differential $d_R$ allows us to define a cochain complex $\{CR_s^n(U,A),d_R^n\}_{n \in \N}$. Then we define \textit{U-local Lie rack cohomology of X with coefficients in A} as the cohomology of the cochain complex $\{CR_s^n(U,A),d_R^n\}_{n \in \N}$.
\section{Lie racks and Leibniz algebras}
\subsection{From Lie racks to Leibniz algebras}
In this section we recall how a Leibniz algebra is canonically associated to a Lie rack.
\begin{proposition}[\cite{Kinyon}] \label{l'espace tangent d'un rack est Leibniz}
Let $X$ be a Lie rack, then $T_{1}X$ is provided with a Leibniz algebra structure.
\end{proposition}
Let $X$ be a Lie rack and denote by $\mathfrak{x}$ the tangent space to $X$ at $1$. The Leibniz algebra structure on $T_1X$ is constructed as follow. The conjugation $\conj$ induces for all $x \in X$ an automorphism of Lie racks $c_{x}:X \to X$ defined by $c_x(y) = x \conj y$. Define for all $x \in X$ the map
$$
Ad_{x} = T_{1}c_{x} \in GL(\xx).
$$ 
The pointed rack structure on $X$ implies that $c_{x \conj y} = c_{x} \circ c_{y} \circ c_{x}^{-1}$ and $c_{1} = id$, hence $Ad:X \to GL(\xx)$ is a morphism of Lie racks. Let $ad = D_1Ad: \xx \to \mathfrak{gl}(\xx)$ the differential of $Ad$ at $1$. Define a bracket $[-,-]$ on $\xx = T_{1}X$ by setting
$$
\lbrack u,v \rbrack = ad(u)(v).
$$
Differentiate the rack identity $x \conj (y \conj z) = (x \conj y) \conj (x \conj z)$ with respect to each variables involves the Leibniz identity for the bracket $[-,-]$ (cf. \cite{Kinyon}).

\begin{example}[Group] Let $G$ be a Lie group. We get in this way the canonical Lie algebra structure on $T_{1}G$.
\end{example}
\begin{example}[Augmented rack] Let $X \fldr{p} G$ be an augmented Lie rack. The linear map $T_{1}X \fldr{T_{1}p} T_{1}G$ is a Lie algebra in the category of linear maps (see \cite{LodayTensor}). This structure induces a Leibniz algebra structure on $T_{1}X$ which is isomorphic to the one induces by the Lie rack structure on $X$.
\end{example}
We remark that a local smooth structure around $1$ is sufficient to provide $T_{1}X$ with a Leibniz algebra structure.
\begin{proposition} \label{l'espace tangent d'un rack local est Leibniz}
Let $X$ be a local Lie rack, then $T_{1}X$ is a Leibniz algebra.
\end{proposition}
\subsection{From $As_p(X)$-modules to Leibniz representations}
Let $X$ be a rack. An $As_p(X)$-module is an abelian group $A$ provided with a morphism of groups $\phi: As_p(X) \to Aut(A)$. By adjointness, this is the same thing as a morphism of pointed racks $\phi: X \to Conj(Aut(A))$.
\begin{definition}
Let $X$ be a Lie rack, a \textbf{smooth} $As(X)$\textbf{-module} is an $As_p(X)$module $A$ where $A$ is an abelian Lie group and $\phi : X \times A \to A$ is smooth.
\end{definition}
\par
Recall that, given a Leibniz algebra $\gg$, a $\gg$-representation $\aa$ is a vector space provided with two linear maps $\lbrack -,- \rbrack_{L} : \gg \otimes \aa \to \aa$ and $\lbrack -,- \rbrack_{R} : \aa \otimes \gg \to \aa$ satisfying the axioms $(LLM), (LML)$ and $(MLL)$ given in section $1$.
\par
There are two particular classes of modules. The first, called \textit{symmetric}, are the modules where $\lbrack -,- \rbrack_{L} = - \lbrack -,- \rbrack_{R}$. The second, called \textit{anti-symmetric}, are the modules where $\lbrack -,- \rbrack_{R} = 0$. Given a Leibniz algebra $\gg$ and $\aa$ a vector space equipped with a morphism of Leibniz algebra $\phi: \gg \to End(\aa)$, we can put two structures of $\gg$-representation on $\aa$. One is \textit{symmetric} and defined by
$$
\lbrack x,a \rbrack_{L} = \phi_{x}(a) \text{ and } \lbrack a,x\rbrack_{R} = -\phi_{x}(a), \ \ \forall x \in \gg, a \in \aa.
$$
The other is \textit{anti-symmetric} and defined by
$$
\lbrack x,a \rbrack_{L} = \phi_{x}(a) \text{ and } \lbrack a,x\rbrack_{R} = 0, \ \ \forall x \in \gg, a \in \aa.
$$
\par
Moreover, given a rack $X$ and $A$ a (smooth) $As(X)$-module, we can put two structures of (smooth) $X$-module on $A$. One is called \textit{symmetric} and defined by
$$
\phi_{x,y}(a) = \phi_{x}(a) \text{ and } \psi_{x,y}(a) = a - \phi_{x \conj y}(a), \ \ \forall x,y \in X, a \in A.
$$
The other is called \textit{anti-symmetric} and defined by
$$
\phi_{x,y}(a) = \phi_{x}(a) \text{ and } \psi_{x,y}(a) = 0, \ \ \forall x,y \in X, a \in A.
$$
\par
These constructions are related to each other because one is the infinitesimal version of the other. Indeed, let $(A,\phi,\psi)$ be a smooth symmetric $X$-module. We have by definition two smooth maps
$$
\phi: X \times X \times A \to A \text{ and } \psi: X \times X \times A \to A
$$
with $\phi_{1,1} = id, \psi_{1,1} = 0$. Thus the differentials of these maps at $(1,1)$ give us two maps
$$
\epsilon: X \times X \to Aut(\aa); \epsilon(x,y) = T_{1}\phi_{x,y} \,
\text{ and } \,
\chi: X \times X \to End(\aa); \chi(x,y) = T_{1}\psi_{x,y}.
$$
These maps are smooth, so we can differentiate them at $(1,1)$ to obtain
$$
T_{(1,1)}\epsilon: \xx \oplus \xx \to End(\aa) \,
\text{ and } \,
T_{(1,1)}\chi: \xx \oplus \xx \to End(\aa).
$$
Then we define two linear maps $\lbrack -,- \rbrack_{L}: \xx \otimes \aa \to \aa$ and $\lbrack -,- \rbrack_{R}: \aa \otimes \xx \to \aa$ by
$$
\lbrack u,m \rbrack_{L} = T_{(1,1)}\epsilon(u,0)(m) \,
\text{ and } \,
\lbrack m,u \rbrack_{R} = T_{(1,1)}\chi(0,u)(m).
$$
\begin{proposition} \label{Lien entre smooth As(X)-modules et Leibniz representations}
Let $X$ be a Lie rack, $\xx$ be its Leibniz algebra, $A$ be an abelian Lie group and $\aa$ be its Lie algebra. If $(A,\phi,\psi)$ is a smooth symmetric (resp. anti-symmetric) $X$-module, then $(\aa,\lbrack -,- \rbrack_{L},\lbrack -,- \rbrack_{R})$ is a symmetric (resp. anti-symmetric) $\xx$-module. 
\end{proposition}
\begin{proof} It is clear that if $(A,\phi,\psi)$ is symmetric then $\lbrack -,- \rbrack_{L} = -\lbrack -,- \rbrack_{R}$, and if $(A,\phi,\psi)$ is anti-symmetric then $[-,-]_{R} = 0$. 
\par
Now let us prove that $[-,-]_{L}$ satisfies the axiom $(LLM)$. By hypothesis on $\phi$, the relation $\phi_{x,y \conj z} \circ \phi_{y,z} = \phi_{x \conj y,x \conj z} \circ \phi_{x,z}$ is true for all $x,y,z \in X$. Taking $z=1$ we obtain $\phi_{x,1} \circ \phi_{y,1} = \phi_{x \conj y,1} \circ \phi_{x,1}$. By differentiating this equality with respect to each variables we find that $[-,-]_{L}$ satisfies the axiom $(LLM)$.
\end{proof}
\subsection{From Lie rack cohomology to Leibniz cohomology}
\begin{proposition} \label{morphisme de la cohomologie de rack vers la cohomologie de Leibniz}
Let $X$ be a Lie rack and let $A$ be a smooth $As(X)$-module. We have morphisms of cochains complexes
$$
CR_{p}^{n}(X,A^{s})_{s} \fldr{\delta^{n}} CL^{n}(\xx,\aa^{s}) \, \text{ and } \,
CR_{p}^{n}(X,A^{a})_{s} \fldr{\delta^{n}} CL^{n}(\xx,\aa^{a}),
$$ 
given by $\delta^{n}(f)(a_1,\dots,a_n) = d^{n}f(1,\dots,1)\big((a_1,0,\dots,0),\dots,(0,\dots,0,a_n)\big)$ (where $d^nf$ is the $n$-th differential of $f$).
\end{proposition}
\begin{proof} Let $f \in CR_{p}^{n}(X,A^{s})$ and $(x_{0}, \dots , x_{n}) \in X^{n+1}$, we want to prove that 
$$
\delta^{n+1}(d_R^nf) = d_L^n(\delta^{n}(f)).
$$
Let $(\gamma_{0}(t_{0}),\dots,\gamma_{n}(t_{n}))$ be a family of paths $\gamma_{i}:]-\epsilon_{i},+\epsilon_{i}[ \to V$ such that $\gamma_{i}(0) = 1\text{ and } \left.\frac{\partial}{\partial s}\right|_{s=0} \gamma_{i}(s) = x_{i}$. Because $f(x_{0},\dots,1,\dots,x_{n}) = 0$, for all $i \in \{1,\dots,n\}$ we have
\begin{align*}
\displaystyle \left.\frac{\partial^{n+1}}{\partial t_{0}\dots \partial t_{n}}\right|_{|t_{i} = 0} \phi_{\gamma_{0}(t_{0}) \conj \dots \conj \gamma_{i}(t_{i})}&(f(\gamma_{0}(t_{0}), \dots, \gamma_{i-1}(t_{i-1}), \gamma_{i+1}(t_{i+1}), \dots, \gamma_{n}(t_{n}))) \\ &= a_{i}.d_{n}(f)(a_{0},\dots,\widehat{a_{i}},\dots,a_{n})
\end{align*}
Moreover for all $i \in \{ 1, \dots, n \}, \displaystyle{\left.\frac{\partial^{n+1}}{\partial t_{0}\dots \partial t_{n}}\right|_{t_{l} = 0}} f(\gamma_{0}(t_{0}),\dots,\gamma_{i}(t_{i}) \conj \gamma_{i+1}(t_{i+1}),\dots,\gamma_{i}(t_{i})\conj \gamma_{n}(t_{n}))$ is equal to
$$
\displaystyle\left.{\frac{\partial}{\partial t_{i}}}\right|_{t_{i}=0} d^{n}f(1,\dots,1)\big((a_{0},0,\dots,0),\dots,(0,\dots,Ad_{\gamma_{i}(t_{i})}(a_{i+1}),\dots,0),\dots, (0,\dots,0,Ad_{\gamma_{i}(t_{i})}(a_{n})\big)
$$
which is equal to
$$
\displaystyle{\sum_{k = i+1}^{n}} d^{n}f(1,\dots,1)\big((a_{0},0,\dots,0),\dots,(0,\dots,[a_{i},a_{k}],\dots,0),\dots,(0,\dots,0,a_{n})\big)
$$
Hence
$$
\begin{array}{ll}
\delta^{n+1}(d_R^nf)(a_{0},\dots,a_{n}) &=  \displaystyle{\sum_{i=0}^{n}} (-1)^{i} \Big(a_{i}.\delta^{n}(f)(a_{0},\dots,\hat{a_{i}},\dots,a_{n}) - \displaystyle{\sum_{k = i+1}^{n}} \delta^{n}f(a_{0},\dots,[a_{i},a_{k}],\dots,a_{n})\Big)\\
					       &=  \displaystyle{\sum_{i=0}^{n}} (-1)^{i} a_{i}.\delta^{n}(f)(a_{0},\dots,\hat{a_{i}},\dots,a_{n}) \\
					       & + \displaystyle{\sum_{ 0 \leq i<k \leq n}} (-1)^{i+1}\delta^{n}f(a_{0},\dots,[a_{i},a_{k}],\dots,a_{n})\\
\end{array}
$$
that is
$$
\delta^{n+1}(d_R^nf) = d_L^n(\delta^{n}(f))
$$
This is exactly the same proof as for the case where $A$ is anti-symmetric.
\end{proof}
We remark that we only need a local cocyle identity around $1$. Thus we have
\begin{proposition} \label{morphisme de la cohomologie de rack locale vers la cohomologie de Leibniz} 
Let $X$ be a Lie rack, let U be a $1$-neighborhood in $X$ and let $A$ be a smooth $As(X)$-module. We have morphisms of cochain complexes
$$
CR_{p}^{n}(U,A^{s}) \fldr{\delta^{n}} CL^{n}(\xx,\aa^{s}) \,
\text{ and } \,
CR_{p}^{n}(U,A^{a}) \fldr{\delta^{n}} CL^{n}(\xx,\aa^{a}),
$$
given by $\delta^{n}(f)(a_0,\dots,a_n) = d^{n}f(1,\dots,1)((a_1,0,\dots,0),\dots,(0,\dots,0,a_n))$.
\end{proposition}
\subsection{From Leibniz cohomology to local Lie rack cohomology}
In this section, we study two cases of Leibniz cocyles integration. This section will be used in the following section to integrate a Leibniz algebra into a local augmented Lie rack.
\par
First, we study the integration of a 1-cocycle in $ZL^{1}(\gg,\aa^{s})$ into a Lie rack 1-cocycle in $ZR^{1}_p(G,\aa^{s})_{s}$, where $G$ is a simply connected Lie group with Lie algebra $\gg$ and $\aa$ a representation of $G$.
\par
Secondly, we use the result of the first part to study the integration of a 2-cocycle in $ZL^{2}(\gg,\aa^{a})$ into a local Lie rack 2-cocycle in $ZR^{2}_p(U,\aa^{a})_{s}$, where $U$ is a $1$-neighborhood in a simply connected Lie group $G$ with Lie algebra $\gg$, and $\aa$ a representation of $G$. It is this second part that we will use to integrate Leibniz algebras. 
\subsubsection{From Leibniz $1$-cocycles to Lie rack $1$-cocycles}
Let $G$ be a simply connected Lie group and $\aa$ a representation of $G$. We want to define a morphism $I^{1}$ from $ZL^{1}(\gg,\aa^{s})$ to $ZR_{p}^{1}(G,\aa^{s})_{s}$ which sends $BL^{1}(\gg,\aa^{s})$ into $BR_{p}^{1}(G,\aa^{s})_{s}$. For this, we put 
$$
I^{1}(\omega)(g) = \int_{\gamma_{g}} \omega^{eq},
$$
where $\omega \in ZL^{1}(\gg,\aa^{s})$, $\gamma: G \times [0,1] \to G$ is a smooth map such that $\gamma_{g}$ is a path from $1$ to $g$, $\gamma_1$ is the constant path equal to $1$, and $\omega^{eq}$ is the closed left equivariant differential form in $\Omega^{1}(G,\aa)$ defined by
$$
\omega^{eq}(g)(m) = g.(\omega(T_{g}L_{g^{-1}}(m))).
$$
By definition, it is clear that $I^1(\omega)(1) = 0$.
\par
For the moment, $I^{1}(\omega)$ depends on $\gamma$, but because $\omega$ is a cocycle and $G$ is simply connected, the dependence with respect to $\gamma$ disappears.  
\begin{proposition}
$I^{1}$ does not depend on $\gamma$.
\end{proposition}
\begin{proof} Let $\gamma,\gamma':G \times [0,1] \to G$ such that $\gamma_{g}(0)=\gamma'_{g}(0)=1$ and $\gamma_{g}(1)=\gamma'_{g}(1)=g$. 
%
%
As $H_{1}(G)=0$, the cycle $\gamma_{g}-\gamma'_{g}$ is a boundary $\partial\sigma_{g}$.
So
\begin{align*}
\int_{\gamma_{g}}\omega^{eq}-\int_{\gamma'_{g}}\omega^{eq}&=\int_{\gamma_{g}-\gamma'_{g}}\omega^{eq} =\int_{\partial\sigma_{g}}\omega^{eq} =\int_{\sigma_{g}}d_{dR}\omega^{eq} =0,
\end{align*}
and $I^{1}$ does not depend on $\gamma$.
\end{proof}
\begin{proposition} \label{I^1 envoit cocycle sur cocycle}
$I^{1}$ sends cocycles to cocycles and coboundaries to coboundaries.
\end{proposition}
\begin{proof} First, let $\omega \in ZL^{1}(\gg,\aa^{s})$, we have
\begin{align*}
d_{R}I(\omega)(g,h) &= g.I(\omega)(h) - I(\omega)(g \conj h) - (g \conj h).I(\omega)(g) + I(\omega)(g)\\
	       			&=g.\int_{\gamma_{h}}\omega^{eq}-\int_{\gamma_{g \conj h}}\omega^{eq}-(g \conj h).\int_{\gamma_{g}}\omega^{eq}+\int_{\gamma_{g}}\omega^{eq}\\
	       			&=\int_{\gamma_{h}}g.\omega^{eq}-\int_{\gamma_{g \conj h}}\omega^{eq}-\int_{\gamma_{g}}(g \conj h).\omega^{eq}+\int_{\gamma_{g}}\omega^{eq}\\
	       			&=\int_{g\gamma_{h}}\omega^{eq}-\int_{\gamma_{g \conj h}}\omega^{eq}-\int_{(g \conj h)\gamma_{g}}\omega^{eq}+\int_{\gamma_{g}}\omega^{eq}\\
	       			&=\int_{g\gamma_{h}-\gamma_{g \conj h}-(g \conj h)\gamma_{g}+\gamma_{g}}\omega^{eq}.
\end{align*}
As $H^{1}(G) = 0$ and $\partial(g\gamma_{h}-\gamma_{g \conj h}-(g \conj h)\gamma_{g}+\gamma_{g}) = 0$, there exists $\gamma_{g,h}: [0,1]^{2} \to G$ such that $\partial\gamma_{g,h} = g\gamma_{h}-\gamma_{g \conj h}-(g \conj h)\gamma_{g}+\gamma_{g}$. Hence, we have 
\begin{align*}
d_{R}I(\omega)(g,h)	&=\int_{\partial \gamma_{g,h}}\omega^{eq}=\int_{\gamma_{g,h}}d_{dR}\omega^{eq}=0.      
\end{align*}
Hence $ZL^{1}(\gg,\aa^{s})$ is sent to $ZR_{p}^{1}(G,\aa^{s})_{s}$.
\vskip 0.2cm
Secondly, let $\omega \in BL^{1}(\gg,\aa^{s})$. There exists $\beta \in \aa$ such that $\omega(m)=m.\beta$. We have
\begin{align*}
I(\omega)(g)&=\int_{\gamma_{g}}\omega^{eq}=\int_{\gamma_{g}}(d_L\beta)^{eq}=\int_{\gamma_{g}}d_{dR}\beta^{eq}=\beta^{eq}(g)-\beta^{eq}(1)=g.\beta - \beta =d_{R}\beta(g).
\end{align*}
Hence $BL^{1}(\gg,\aa^{s})$ is sent to $BR_{p}^{1}(G,\aa^{s})_{s}$.
\end{proof}
\begin{proposition}
$I^{1}$ is a left inverse for $\delta^{1}$.
\end{proposition}
\begin{proof} Let $\omega \in ZL^{1}(\gg,\aa^{s})$. Let $\varphi:U \to \gg$ be a local chart around $1$ such that $\varphi(1)=0$ and $d\varphi^{-1}(0)=id$. We define for $x \in \gg$ the smooth map  $\alpha_{x}: ]-\epsilon,+\epsilon[ \to U$ by setting $\alpha_{x}(s) = \varphi^{-1}(sx)$, and we define for all $s \in ]-\epsilon,+\epsilon[$ the smooth map $\gamma_{\alpha_{x}(s)}:[0,1] \to U$ by setting $\gamma_{\alpha_{x}(s)}(t) = \varphi^{-1}(tsx)$. We have
\begin{align*}
\delta^{1}(I^{1}(\omega))(x)&= \dpart{s}I^{1}(\omega)(\alpha_{x}(s)) = \dpart{s}\int_{\gamma_{\alpha_{x}(s)}}\omega^{eq} =\dpart{s}\int_{[0,1]}\gamma_{\alpha_{x}(s)}^{*}\omega^{eq}\\
 &=\dpart{s}\int_{[0,1]}\omega^{eq}(\gamma_{\alpha_{x}(s)}(t))(\dpart{t}\gamma_{\alpha_{x}(s)}(t))dt\\
%
%
	       			      &=\int_{[0,1]}\dpart{s}\omega^{eq}(\varphi^{-1}(tsx))(sx)dt =\int_{[0,1]}\dpart{s}(\varphi^{-1})^{*}\omega^{eq}(tsx)(sx)dt\\
	       			      &=\int_{[0,1]}\dpart{s}s(\varphi^{-1})^{*}\omega^{eq}(tsx)(x)dt 
				      =\omega(x)\int_{[0,1]}dt\\
	       			      &=\omega(x).
\end{align*}
Hence $\delta^{1} \circ I^{1} = id$.
\end{proof}
\begin{remark} \label{I^1 est un cocycle de groupe}
In fact, $I^1(\omega)$ is also a Lie group 1-cocycle. Indeed, the formula to define $I^1(\omega)$ is the same as the one defined by K.H. Neeb in Section 3 of \cite{Neeb2}, and in this article he shows that $I^1(\omega)$ is a group cocycle. The following calculation shows that this group cocycle identity satisfied by $I^1(\omega)$ implies the rack cocycle identity satisfied by $I^1(\omega)$. Indeed, $I^1(\omega)$ is a group cocycle, thus we have
$$
d_{Gp}I^{1}(\omega)(g,h) - d_{Gp}I^1(\omega)(g \conj h,g) = 0.
$$
Moreover $d_{Gp}I^{1}(\omega)(g,h) - d_{Gp}I^1(\omega)(g \conj h,g) = d_RI^1(\omega)(g,h)$, thus $d_RI^1(\omega)(g,h) = 0$, and we see clearly that the rack cocycle identity is implied by the group cocycle identity. We will use this remark in Proposition \ref{identité fantôme de I^2}.  
\end{remark}
\subsubsection{From Leibniz $2$-cocycles to Lie local rack $2$-cocycles}
Let $G$ be a simply connected Lie group, let $U$ be a $1$-neigbourhood in $G$ such that $\log$ is defined on $U$ and let $\aa$ be a representation of $G$. In Proposition \ref{morphisme de la cohomologie de rack locale vers la cohomologie de Leibniz} we have defined for all $n \in \N$ the maps 
$$
HR_{s}^{n}(U,\aa^{a}) \lfldr{[\delta^{n}]} HL^{n}(\gg,\aa^{a}).
$$
In the next section, we will see that a Leibniz algebra can be integrated into a local Lie rack since the morphism  $[\delta^{2}]$ is surjective. More precisely, if we can construct a left inverse for $[\delta^{2}]$, then it gives us an explicit method to construct the local Lie rack which integrates the Leibniz algebra.
\par
In this section, we define a morphism $[I^{2}]$ from $HL^{2}(\gg,\aa^{a})$ to $HR_{s}^{2}(U,\aa^{a})$, and we show that it is a left inverse for $[\delta^{2}]$. To construct the map $[I^{2}]$, we adapt an integration method of Lie algebra cocycles into Lie group cocycles by integration over simplex. This method is due to W.T. Van Est (\cite{VanEst}) and used by K.H. Neeb (\cite{Neeb1,Neeb2}) for the infinite dimensional case. 
\begin{center}
\textbf{Definition of $I^{2}$}
\end{center}
\par
We want to define a map from $ZL^{2}(\gg,\aa^{a})$ to $ZR_{p}^{2}(U,\aa^{a})_{s}$ such that $BL^{2}(\gg,\aa^{a})$ is sent to $BR_{p}^{2}(U,\aa^{a})_{s}$. In the previous section, we have integrated a Leibniz $1$-cocycle on a Lie algebra $\gg$ with coefficients in a symmetric module $\aa^{s}$. In Proposition \ref{isomorphisme entre CLn et CLn-1}, we have shown that there is an isomorphism between $CL^{2}(\gg,\aa^{a})$ and $CL^{1}(\gg,Hom(\gg,\aa)^{s})$, which sends $ZL^{2}(\gg,\aa^{a})$ to $ZL^{1}(\gg,Hom(\gg,\aa)^{s})$ and $BL^{2}(\gg,\aa^{a})$ to $BL^{1}(\gg,Hom(\gg,\aa)^{s})$. Hence, we can define a map
$$
I : ZL^{2}(\gg,\aa^{a}) \to ZR_{p}^{1}(G,Hom(\gg,\aa)^{s})_{s},
$$
which sends $BL^{2}(\gg,\aa^{a})$ into $BR_{p}^{1}(G,Hom(\gg,\aa)^{s})_{s}$.
This is the composition
$$
ZL^{2}(\gg,\aa^{a}) \fldr{\tau^{2}} ZL^{1}(\gg,Hom(\gg,\aa)^{s}) \fldr{I^{1}} ZR_{p}^{1}(G,Hom(\gg,\aa)^{s})_{s}.
$$
Now, we want to define a map from $ZR_{p}^{1}(G,Hom(\gg,\aa)^{s})_{s}$ to $ZR_{p}^{2}(U,\aa^{a})$. Let $\beta \in CR_{p}^{1}(G,Hom(\gg,\aa)^{s})_{s}$, $\beta$ has values in the representation $Hom(\gg,\aa)$, so for all $g \in G$, we can consider the equivariant differential form $\beta(g)^{eq} \in \Omega^{1}(G,\aa)$ defined by 
$$
\beta(g)^{eq}(h)(m) := h.(\beta(g)(T_{h}L_{h^{-1}}(m))).
$$
Then we define an element in $CR_{p}^{2}(U,\aa^{a})$ by setting
$$
f(g,h) = \int_{\gamma_{g \conj h}}(\beta(g))^{eq},
$$
where $\gamma: G \times [0,1] \to G$ is a smooth map such that for all $g \in G$, $\gamma_{g}$ is a path from $1$ to $g$ in $G$ and $\gamma_1 = 1$.\\
For the moment, an element of $ZR_{p}^{1}(G,Hom(\gg,\aa)^{s})_{s}$ is not necessarily sent to an element of $ZR_{p}^{2}(U,\aa^{a})_{s}$. To reach our goal, we have to specify the map $\gamma$, and we define it by setting
$$
\gamma_{g}(s) = \exp(s\log(g)).
$$
Then, we define $I^{2}: ZL^{2}(\gg,\aa^{a}) \to CR_{p}^{2}(U,\aa^{a})_{s}$ by setting for all $(g,h) \in U_{2-loc}$ (cf. notation in Section $2.4$) 
$$
I^{2}(\omega)(g,h) = \int_{\gamma_{g \conj h}}(I(\omega)(g))^{eq}.
$$
By definition, it is clear that $I^2(\omega)(g,1) = I^2(\omega)(1,g) = 0$.
\begin{center}
\textbf{Properties of $I^{2}$}
\end{center}
\begin{proposition} \label{I^2 envoit cocyle sur cocycle}
$I^{2}$ sends $ZL^{2}(\gg,\aa^{a})$ into $ZR_{p}^{2}(U,\aa^{a})_{s}$.  
\end{proposition}
\begin{proof} Let $\omega \in ZL^{2}(\gg,\aa^{a})$ and $(g,h,k) \in U_{3-loc}$. We have
\begin{align*}
d_{R}(I^{2}(\omega))(g,h,k) &= g.I^{2}(\omega)(h,k) - I^{2}(\omega)(g \conj h,g \conj k) - (g \conj h).I^{2}(\omega)(g,k) + I^{2}(\omega)(g,h \conj k)\\
			     		    &= \int_{\gamma_{h \conj k}}g.((I(\omega)(h))^{eq}) - \int_{\gamma_{g \conj (h \conj k)}}(I(\omega)(g \conj h))^{eq} - \int_{\gamma_{g \conj k}}(g \conj h).((I(\omega)(g))^{eq}) \\
					    & + \int_{\gamma_{g \conj (h \conj k)}}(I(\omega)(g))^{eq}.
\end{align*}
For all $g \in G$ we have $g.(\omega^{eq}) = c_{g}^{*}((g.\omega)^{eq})$, thus
%
%
\begin{align*}					  
d_{R}(I^{2}(\omega))(g,h,k) 
			     		   &= \int_{c_{g} \circ \gamma_{h \conj k}}(g.I(\omega)(h)^{eq} - \int_{\gamma_{g \conj (h \conj k)}}I(\omega)(g \conj h)^{eq} - \int_{c_{g \conj h} \circ \gamma_{g \conj k}}((g \conj h).I(\omega)(g))^{eq}\\
					   &+\int_{\gamma_{g \conj (h \conj k)}}I(\omega)(g)^{eq}.
\end{align*}
By naturality of the exponantial and the logarithm, we have for all $(g,h) \in U_{2-loc} \, \gamma_{g \rhd h} = g \rhd \gamma_h$, thus
%
%
\begin{align*}
d_{R}(I^{2}(\omega))(g,h,k) 
			     		    &= \int_{\gamma_{g \conj (h \conj k)}}d_{R}(I(\omega))(g,h) = 0.   
\end{align*}
Hence $ZL^{2}(\gg,\aa^{a})$ is sent to $ZR_{p}^{2}(U,\aa^{a})_{s}$.
\end{proof}
\begin{proposition}
$I^{2}$ sends $BL^{2}(\gg,\aa^{a})$ into $BR_{p}^{2}(U,\aa^{a})_{s}$.
\end{proposition}
\begin{proof}  Let $\omega \in BL^{2}(\gg,\aa^{a})$, there exists an element $\beta \in CL^{1}(\gg,\aa^{a})$ such that $\omega = d_L\beta$. By definition $\displaystyle I(\omega)(g) = \int_{\gamma_{g}}(\tau^{2}(d_L\beta))^{eq}$, and because $\{\tau^{n}\}_{n \in \N}$ is a morphism of cochain complexes
\begin{align*}
I(\omega)(g) &= I^{1}(\tau^{2}(\omega))(g) = \int_{\gamma_{g}}(\tau^{2}(\omega))^{eq} = \int_{\gamma_{g}}(\tau^{2}(d_L\beta))^{eq}. 
\end{align*}
%
%
%
Let $(g,h) \in U_{2-loc}$. Using the same kind of computation as in the proof of Proposition \ref{I^2 envoit cocyle sur cocycle} we find
\begin{align*}
I_{2}(\omega)(g,h) &= \int_{\gamma_{g \conj h}} (I(\omega)(g))^{eq} = g.\int_{\gamma_{h}} \beta^{eq} - \int_{\gamma_{g \conj h}} \beta^{eq} = d_{R}(I^{1}(\beta))(g,h)
\end{align*}
Hence $BL^{2}(\gg,\aa^{a})$ is sent to $BR^{2}_{p}(U,\aa)_{s}$.
\end{proof}
\begin{proposition} \label{I2 inverse à gauche de D2}
$I^{2}$ is a left inverse for $\delta^{2}$.
\end{proposition}
\begin{proof} Let $x,y \in \gg$, and $I_{x}$ (resp $I_{y}$) be an interval in $\mathbb{R}$ such that $\epsilon_{x}(s)=\exp(sx)$ (resp $\epsilon_{y}(s)=\exp(sy)$) be defined for all $s \in I_{x}$ (resp for all $s \in I_{y}$). The map $\epsilon_{x} \conj \epsilon_{y}:I_{x} \times I_{y} \to G$ is continuous, thus there exists $W$ an open subset of $I_{x} \times I_{y}$ such that $(\epsilon_{x} \conj \epsilon_{y})(W) \subseteq U$. Hence there exists an interval $J \subseteq I_{x} \cap I_{y}$ such that $\epsilon_{x}(s) \conj \epsilon_{y}(t) \in U$ for all $(s,t) \in J \times J$.\\
We have to prove 
$$
\delta^{2} \circ I^{2} = id.
$$
Let $\omega \in ZL^{2}(\gg,\aa^{a})$. By definition
\begin{align*}
\delta^{2}(I^{2}(\omega))(x,y) &= \left.\frac{\partial^{2}}{\partial s \partial t}\right|_{s,t=0}I^{2}(\omega)(\epsilon_{x}(s),\epsilon_{y}(s)) = \left.\frac{\partial^{2}}{\partial s \partial t}\right|_{s,t=0}\int_{\gamma_{\epsilon_{x}(s) \conj \epsilon_{y}(t)}} (I(\omega)(\epsilon_{x}(s)))^{eq}\\
			  		 &= \left.\frac{\partial}{\partial s}\right|_{s=0}(\left.\frac{\partial}{\partial 
			      t}\right|_{t=0}\int_{\gamma_{\epsilon_{y}(t)}} c_{\epsilon_{x}(s)}^{*}(I(\omega)(\epsilon_{x}(s)))^{eq}).
\end{align*}
First, we compute $
\left.\frac{\partial}{\partial t}\right|_{t=0}\int_{\gamma_{\epsilon_{y}(t)}} c_{\epsilon_{x}(s)}^{*}(I(\omega)(\epsilon_{x}(s)))^{eq}$. For the sake of clarity, we put $\alpha = c_{\epsilon_{x}(s)}^{*}(I(\omega)(\epsilon_{x}(s)))^{eq}$ and $\beta_{t} = \gamma_{\epsilon_{y}(t)}$. We have
\begin{align*}
\left.\frac{\partial}{\partial t}\right|_{t=0}\int_{\beta} \alpha &= \left.\frac{\partial}{\partial t}\right|_{t=0}\int_{[0,1]} \beta^{*}\alpha = \left.\frac{\partial}{\partial t}\right|_{t=0}\int_{[0,1]} f_{t}(r)dr = \int_{[0,1]}\left.\frac{\partial}{\partial t}\right|_{t=0} f_{t}(r)dr,
\end{align*}
where $f_{t}(r) = \alpha(\beta_{t}(r))(\beta_{t}'(r))$.\\
We have 
$$
\left.\frac{\partial}{\partial t}\right|_{t=0}f_{t}(r) = (\left.\frac{\partial}{\partial t}\right|_{t=0}\alpha(\beta_{t}(r)))\beta_{0}'(r) + (\alpha(\beta_{0}(r)))(\left.\frac{\partial}{\partial t}\right|_{t=0}\beta_{t}'(r)).
$$
Moreover, $\alpha(\beta_{0}(r)) = \alpha(1), \, \beta_{0}'(r) = 0$, and $\left.\frac{\partial}{\partial t}\right|_{t=0}\beta_{t}'(r) = y$. So $\left.\frac{\partial}{\partial t}\right|_{t=0}\int_{\beta} \alpha = \int_{[0,1]}\alpha(1)(x)dr = \alpha(1)(y)$ and $
\delta^{2}(I^{2}(\omega))(x,y) = \left.\frac{\partial}{\partial s}\right|_{s=0}(c_{\epsilon_{x}(s)}^{*}(I(\omega)(\epsilon_{x}(s)))^{eq})(1)(y)$.
Furthermore we have
\begin{align*}
c_{\epsilon_{x}(s)}^{*}(I(\omega)(\epsilon_{x}(s))^{eq})(1)(y) &= (I(\omega)(\epsilon_{x}(s)))^{eq}(c_{\epsilon_{x}(s)}(1))(Ad_{\epsilon_{x}(s)}(y))\\
						       						&= I(\omega)(\epsilon_{x}(s))(Ad_{\epsilon_{x}(s)}(y))\\
						       						&= (\int_{\gamma_{\epsilon_{x}(s)}}\tau^{2}(\omega)^{eq})(Ad_{\epsilon_{x}(s)}(y)).
\end{align*}
If we put $\int_{\gamma_{\epsilon_{x}(s)}}\tau^{2}(\omega)^{eq} = \sigma(s)$ and $Ad_{\epsilon_{x}(s)}(y) = \lambda(s)$, we have
\begin{align*}
\left.\frac{\partial}{\partial s}\right|_{s=0}((\int_{\gamma_{\epsilon_{x}(s)}}\tau^{2}(\omega)^{eq})(Ad_{\epsilon_{x}(s)}(y))) &= \left.\frac{\partial}{\partial s}\right|_{s=0}\sigma(s)(\lambda(s)) = \sigma'(0)(\lambda(0))+\sigma(0)(\lambda'(0)). 
\end{align*}
We have $\sigma(0) = 0, \, \lambda(0) = y$, and $\sigma'(0) = \tau^{2}(\omega)(x)$. Thus
$$
\left.\frac{\partial}{\partial s}\right|_{s=0}\Big(\big(\int_{\gamma_{\epsilon_{x}(s)}}\tau^{2}(\omega)^{eq}\big)\big(Ad_{\epsilon_{x}(s)}(y)\big)\Big)=\tau^{2}(\omega)(x)(y).
$$
Hence $\delta^{2}(I^{2}(\omega))(x,y) = \omega(x,y)$.
\end{proof}
\begin{remark} Suppose that we have a Leibniz 2-cocycle $\omega$ which is also a Lie 2-cocycle. In this case, we can integrate $\omega$ into a local Lie rack cocycle, but also into a local Lie group cocycle (cf. Introduction). Then it is natural to ask if the two constructions are related to each other.
\end{remark}
\begin{proposition}
Let $G$ be a Lie group, $\gg$ be its Lie algebra, $\aa$ be a representation of $G$, $\omega \in \Lambda^{2}(\gg,\aa)$ and $\gamma_1,\gamma_2$ smooth paths in $G$ pointed in $1$. Then 
$$
\int_{\gamma_{1}}\big(\int_{\gamma_{2}} (\tau^{2}(\omega))^{eq}\big)^{eq} = \int_{\gamma_{1}\gamma_{2}} \omega^{eq}
$$
where $\gamma_1\gamma_2 : [0,1]^{2} \to G; (s,t) \mapsto \gamma_1(t)\gamma_2(s)$. 
\end{proposition}
\begin{proof} On the one hand, we have
\begin{align*}
\int_{\gamma_{1}\gamma_{2}} \omega^{eq} &= \int_{[0,1]^2}(\gamma_1\gamma_2)^*\omega^{eq} = \int_{[0,1]^2}\omega^{eq}(\gamma_1\gamma_2)(\dpartt{s}\gamma_1(t)\gamma_2(s),\dpartt{t}\gamma_1(t)\gamma_2(s))dsdt, 
\end{align*}
and this expression is equal to 
\begin{eqnarray}
\int_{[0,1]^2} \gamma_1(t)\gamma_2(s).\omega\big(d_{\gamma_2(s)}L_{\gamma_2(s)^{-1}}(\dpartt{s}\gamma_2(s)),Ad_{\gamma_2(s)^{-1}}(d_{\gamma_1(t)}L_{\gamma_1(t)^{-1}}(\dpartt{t}\gamma_1(t)))\big)dsdt. \label{expression 1}
\end{eqnarray}
On the other hand, we have
\begin{align*}
\int_{\gamma_{1}}\big(\int_{\gamma_{2}} (\tau^{2}(\omega))^{eq}\big)^{eq} &= \int_{[0,1]}\gamma_1^*(\int_{\gamma_2}(\tau^2(\omega)^{eq})^{eq}) \\
														     &= \int_{[0,1]}\gamma_1(t).(\int_{[0,1]}(\tau^2(\omega))^{eq}(\gamma_2(s))(\dpartt{s}\gamma_2(s)))(d_{\gamma_1(t)}L_{\gamma_1(t)^{-1}}(\dpartt{t}\gamma_1(t)))dt.
\end{align*}
and this expression is equal to
$$
\int_{[0,1]}\gamma_1(t).(\int_{[0,1]}\gamma_2(s).\omega(d_{\gamma_2(s)}L_{\gamma_2(s)^{-1}}(\dpartt{s}\gamma_2(s)), Ad_{\gamma_2(s)^{-1}}( \cdot ))ds)(d_{\gamma_1(t)}L_{\gamma_1(t)^{-1}}(\dpartt{t}\gamma_1(t)))dt.
$$
Using the Fubini theorem we see that this expression is equal to \eqref{expression 1}.
\end{proof}
If we apply this result to the case where $\gamma_{1}(s) = \gamma_{g \conj h}(s) = \exp(s\log(g \conj h))$ and $\gamma_{2}(s) = \gamma_{g}(s) = \exp(s\log(g))$ for $(g,h) \in U_{2-loc}$, then we obtain the folowing corollary.
\begin{corollary} \label{cocycle de rack venant d'un cocycle de groupe}
If $\omega \in ZL^{2}(\gg,\aa^{a}) \cap Z^{2}(\gg,\aa)$, then for all $g,h \in U_{2-loc}$ such that $gh \in U_{2-loc}$ we have 
\begin{align}
I^{2}(\omega)(g,h) = \iota^{2}(\omega)(g,h) - \iota^{2}(\omega)(g \conj h,g), \label{formule reliant I^{2} et iota^2}
\end{align}
with
$$
\iota^2(\omega)(g,h) = \int_{\gamma_{g,h}} \omega^{eq},
$$
and where $\gamma_{g,h}$ is a smooth singular $2$-chain in $G$ such that $\partial\gamma_{g,h} = \gamma_g - \gamma_{gh} + g\gamma_h$.
\end{corollary}
\par
We can remark that $I^{2}$ is more than a local Lie rack cocycle. Precisely, if $\omega$ is in $ZL^{2}(\gg,\aa^{a})$ then the local rack cocycle identity satisfied by $I^{2}(\omega)$, comes from another identity satisfied by $I^{2}(\omega)$. Indeed, $I^{2}$ is defined using $I$, and to verify that $I^{2}$ sends Leibniz cocycles into local rack cocycles, we have used Proposition \ref{I^1 envoit cocycle sur cocycle}. This proposition establishes that $I^1$ sends Lie cocycles into rack cocycle. But, we have remarked (Remark \ref{I^1 est un cocycle de groupe}) that the rack cocycle identity satisfied by $I^1(\omega)$, comes from the group cocycle identity. Hence, we can think that we forgot structure on $I^{2}(\omega)$. The following proposition points out the identity satisfied by $I^{2}(\omega)$ which induced the local rack identity.
\begin{proposition} \label{identité fantôme de I^2} 
If $\omega \in ZL^{2}(\gg,\aa^{a})$, then $I^{2}(\omega)$ satisfies the identity
$$
g.I^{2}(\omega)(h,k) - I^{2}(\omega)(gh,k) + I^{2}(\omega)(g,h \conj k) = 0, \ \ \forall (g,h,k) \in U_{3-loc}. 
$$
Moreover, this identity induces the local rack cocycle identity. 
\end{proposition}
\begin{proof} Let $\omega \in ZL^{2}(\gg,\aa^{a})$ and $(g,h,k) \in U_{3-loc}$. Using the same kind of computation as in the proof of Proposition \ref{I^2 envoit cocyle sur cocycle} we find :
\begin{align*}
g.I^{2}(\omega)(h,k) - I^{2}(\omega)(gh,k) + I^{2}(\omega)(g,h \conj k) &= \int_{\gamma_{g \conj (h \conj k)}}d(I(\omega))(g,h) = 0.
\end{align*}
Let $(g,h,k) \in U_{3-loc}$. Denote the expression $g.I^{2}(\omega)(h,k) - I^{2}(\omega)(gh,k) + I^{2}(\omega)(g,h \conj k)$ by $b(I^{2}(\omega))(g,h,k))$. Inserting $-I^2(\omega)(gh,k) + I^2(\omega)(gh,k)$ in the formula for $d_{R}(I^{2}(\omega))(g,h,k)$, we find 
$$
d_{R}(I^{2}(\omega))(g,h,k) = b(I^{2}(\omega))(g,h,k))-b(I^{2}(\omega)(g \rhd h,g,k)).
$$ 

\end{proof}
We will see in the next section that this identity makes it possible to integrate a Leibniz algebra into a local augmented Lie rack.
\subsection{From Leibniz algebras to local Lie racks}
In this section, we present the main theorem of this article. In Proposition \ref{l'espace tangent d'un rack est Leibniz} we have seen that the tangent space at $1$ of a (local) Lie rack is provided with a Leibniz algebra structure. Conversely, we now show that every Leibniz algebra can be integrated into an augmented local Lie rack. Our construction is explicit, and by this construction, a Lie algebra is integrated into a Lie group. Conversely, we show that an augmented local Lie rack whose tangent space at $1$ is a Lie algebra is necessarily a (local) Lie group. That is, there is a structure of Lie group on this augmented local Lie rack, and the conjugation on the augmented local Lie rack is the conjugation in the group.
\vskip 0.2cm
The idea of the proof is simple and uses the knowledge of the Lie's first theorem and Lie's second theorem. Let $\gg$ be a Leibniz algebra. First, we decompose the vector space $\gg$ into a direct sum of Leibniz algebras $\gg_0$ and $\aa$ that we know how to integrate. As we know the theorem for Lie subalgebras of endomorphisms of a finite dimensional vector space $V$, the factors are integrable if $\gg$ is isomorphic (as a vector space) to an abelian extension of a Lie subalgebra $\gg_0$ of $End(V)$ by a $\gg_0$-representation $\aa$. Hence $\gg$ is isomorphic to $\aa \oplus_{\omega} \gg_0$,  the Leibniz algebra $\aa$ is abelian so becomes integrated into $\aa$, and $\gg_0$ is a Lie subalgebra of $End(V)$ so becomes integrated into a simply connected Lie subgroup $G_0$ of $Aut(V)$. Now, we have to understand how to patch $\aa$ and $G_0$. That is, we have to understand how the gluing data $\omega$ becomes integrated into a gluing data $f$ between $\aa$ and $G_0$. It is the local Lie rack cocycle $I^{2}(\omega)$, constructed in the preceding section, which answers this question. Hence, we showed that a Leibniz algebra $\gg$ becomes integrated into a local Lie rack of the form $\aa \times_{f} G_0$.
\vskip 0.2cm
Let $\gg$ be a Leibniz algebra, there are several ways to see $\gg$ as an abelian extension of a Lie subalgebra $\gg_0$ of $End(V)$ by a $\gg_0$-representation $\aa$. Here, we take the abelian extension associated to the (left) center of $\gg$. By definition the left center is 
$$
Z_L(\gg) = \{x \in \gg \, | \, [x,y] = 0 \ \ \forall y \in \gg\}.
$$
The left center $Z_L(\gg)$ is an ideal in $\gg$ and we can consider the quotient of $\gg$ by $Z_L(\gg)$. By definition, $Z_L(\gg)$ is the kernel of the \textit{adjoint representation} $ad_L:\gg \to End(\gg), x \mapsto [x,-]$. Thus this quotient is isomorphic to a Lie subalgebra of $End(\gg)$. We denote this quotient by $\gg_0$. Hence, to a Leibniz algebra $\gg$ there is a canonical abelian extension given by
$$
Z_L(\gg) \fldri{i} \gg \fldrs{p} \gg_0.
$$
This extension gives a structure of $\gg_0$-representation to $Z_L(\gg)$, and by definition of $Z_L(\gg)$, this representation is anti-symmetric. The equivalence class of this extension is characterised by a cohomology class in $HL^{2}(\gg_0,Z_L(\gg))$. Hence there is $\omega \in ZL^{2}(\gg_0,Z_L(\gg))$ such that the abelian extension $Z_L(\gg) \fldri{i} \gg \fldrs{p} \gg_0$ is equivalent to 
$$
Z_L(\gg) \fldri{i} \gg_0 \oplus_{\omega} Z_L(\gg) \fldrs{\pi} \gg_0.
$$
Here $\gg_0$ is a Lie subalgebra of $End(\gg)$, so becomes integrated into a simply connected Lie subgroup $G_0$ of $Aut(\gg)$, and $Z_L(\gg)$ is an abelian Lie algebra, so becomes integrated into itself. $Z_L(\gg)$ is a $\gg_0$-representation (in the sense of Lie algebra) and $G_0$ is simply connected, thus by the Lie's second theorem, $Z_L(\gg)$ is a smooth $G_0$-module (in the Lie group sense) and we can provide $Z_L(\gg)$ with an anti-symmetric smooth $G_0$-module structure. The cocycle $\omega \in ZL^2(\gg,Z_L(\gg))$ becomes integrated into the local Lie rack cocycle $I^{2}(\omega) \in ZR_{p}^{2}(G_0,Z_L(\gg))_s$, and we can put on the cartesian product $G_0 \times Z_L(\gg)$ a structure of local Lie rack by setting
$$
(g,a) \conj (h,b) = (g \conj h,\phi_{g,h}(b) + \psi_{g,h}(a) + I^2(\omega)(g,h)),
$$ 
where $\phi_{g,h}(b) = g.b$ and $\psi_{g,h}(a) = 0$. That is we have
$$
(g,a) \conj (h,b) = (g \conj h, g.b + I^{2}(\omega)(g,h))).
$$
It is clear by construction that this local Lie rack has its tangent space at $1$ provided with a Leibniz algebra structure isomorphic to $\gg$. Finally, we have shown the following theorem.
\begin{theoreme} \label{theoreme_un}
Every Leibniz algebra $\gg$ can be integrated into a local Lie rack of the form 
$$
G_0 \times_{I^2(\omega)} \aa^{a},
$$
with conjugation
\begin{eqnarray}
(g,a) \conj (h,b) = (g \conj h, g.b + I^{2}(\omega)(g,h)), \label{conjugaison dans l'extension anti-symetrique}
\end{eqnarray}
and neutral element $(1,0)$, where $G_0$ is a Lie group, $\aa$ a $G_0$-module and $\omega \in ZL^2(\gg_0,\aa^{a})$. Conversely, every local Lie rack of this form has its tangent space at $1$ provides with a Leibniz algebra structure.
\end{theoreme}
We ask more in our original problem. Indeed, we ask that, using the same procedure, a Lie algebra becomes integrated into a Lie group. That is, we have to show that when $\gg$ is a Lie algebra, then $G_0 \times Z_L(\gg)$ is provided with a Lie group structure, and the conjugation on $G_0 \times_{I^2(\omega)} Z_L(\gg)$ is induced by the rack product in $Conj(G_0 \times Z_L(\gg))$.
\vskip 0.2cm
Let $\gg$ be a Lie algebra, the left center $Z_L(\gg)$ is equal to the center $Z(\gg)$. The abelian extension $Z_L(\gg) \fldri{i} \gg \fldrs{p} \gg_0$ provides $Z_L(\gg)$ with an anti-symmetric structure but also a symmetric structure, so a trivial structure. This extension becomes a central extension and the cocycle $\omega \in ZL^2(\gg_0,Z(\gg))$ is also in $Z^{2}(\gg_0,Z(\gg))$. On the hand, with $\omega$ we can construct a local Lie rack cocycle $I^2(\omega)$, and on the other hand, we can construct a Lie group cocycle $\iota^2(\omega)$. Hence, using the formula (\ref{formule reliant I^{2} et iota^2}) relating $I^{2}(\omega)$ and $\iota^2(\omega)$, the conjugation in $G_0 \times_{I^2(\omega)} Z(\gg)$ can be written
$$
(g,a) \conj (h,b) = (g \conj h, I^2(\omega)(g,h)) = (g \conj h, \iota^2(\omega)(g,h) - \iota^2(\omega)(g \conj h,g)),
$$
and a easy calculation shows that this is the formula for the conjugation in the group $G_0 \times_{\iota^2(\omega)} Z(\gg)$, where the product is defined by
$$
(g,a)(h,b) = (gh, \iota^2(g,h)).
$$
\vskip 0.2cm
Conversely, suppose that a local Lie rack of the form $G_0 \times_{I^2(\omega)} \aa^{a}$ has its tangent space at $1$, $\gg_0 \oplus_{\omega} \aa^{a}$, provided with a Lie algebra structure. Necessarily, $\aa$ is a trivial $\gg_0$-representation and $\omega \in Z^2(\gg_0,\aa)$. Hence, as before we have the formula (\ref{formule reliant I^{2} et iota^2}) relating $I^2(\omega)$ and $\iota^2(\omega)$ and the conjugation defined by the formula $\eqref{conjugaison dans l'extension anti-symetrique}$ is induced by the conjugation coming from the group structure on $G_0 \times_{\iota^{2}(\omega)} \aa$. 
Finally, we have the following refinement of Theorem \ref{theoreme_un}.
\begin{theoreme} \label{theoreme principal}
Every Leibniz algebra $\gg$ can be integrated into a local Lie rack of the form 
$$
G_0 \times_{I^2(\omega)} \aa^{a},
$$
with conjugation
\begin{align}
(g,a) \conj (h,b) = (g \conj h, g.b + I^{2}(\omega)(g,h)), \label{conjugation} 
\end{align}
and neutral element $(1,0)$, where $G_0$ is a Lie group, $\aa$ a representation of $G_0$ and $\omega \in ZL^2(\gg_0,\aa^{a})$. Conversely, every local Lie rack of this form has its tangent space at $1$ provided with a Leibniz algebra structure.
\par
Moreover, in the special case where $\gg$ is a Lie algebra, the above construction provides $G_0 \times_{I^2(\omega)} \aa^a$ with a rack product coming from the conjugation in a Lie group. Conversely, if the tangent space at $1$ of $G_0 \times_{I^2(\omega)} \aa^{a}$ is a Lie algebra, then $G_0 \times_{I^2(\omega)} \aa^{a}$ can be provided with a Lie group structure, and the conjugation induced by the Lie group structure is the one defined by \eqref{conjugation}.
\end{theoreme} 
\subsection{From Leibniz algebras to local augmented Lie racks}
Let $\gg_0$ be a Lie algebra, $\aa$ a $\gg$-representation and $\omega \in ZL^{2}(\gg_0,\aa^{a})$. In Proposition \ref{I^2 envoit cocyle sur cocycle}, we showed that $I^2(\omega)$ is a local Lie rack cocycle. We showed also that it satisfies the identity
\begin{eqnarray}
g.I^2(\omega)(h,k) - I^2(\omega)(gh,k) + I^2(\omega)(g, h \conj k) = 0 \label{cocycle augmented rack identity}
\end{eqnarray}
for all $(g,h,k) \in U_{3-loc}$. The natural question is to know which algebraic structure on $G_0 \times_{I^2(\omega)} \aa^{a}$ is encoded by this identity. We will see that the answer is the structure of a local augmented Lie rack.
\begin{definition}
Let $G$ be a group. A \textbf{local} $G$\textbf{-set} is a set $X$ provides with a map $\rho$ defined on a subset $\Omega$ of $G \times X$ with values in $X$ such that the followings axioms are satisfied
\begin{enumerate}
\item If $(h,x),(gh,x),(g,\rho(h,x)) \in \Omega$, then $\rho(g,\rho(h,x)) = \rho(gh,x)$.
\item For all $x \in X$, we have $(1,x) \in \Omega$ and $\rho(1,x) = x$.
\end{enumerate}
A \textbf{local topological} (resp.(\textbf{smooth})) $G$\textbf{-set} is a topological set (resp. smooth manifold) $X$ with a structure of a local $G$-set where $\Omega$ is an open subset of $X$ and $\rho : \Omega \to X$ is continuous (resp. smooth). A \textbf{fixed point} is an element $x_0 \in X$ such that for all $g \in G, (g,x_0) \in \Omega$ and $\rho(g,x_0) = x_0$. 
\end{definition}
In the following proposition, we show that the identity $\eqref{cocycle augmented rack identity}$ provides $G_0 \times_{I^2(\omega)} \aa^{a}$  with a structure of a local $G_0$-set.
\begin{proposition}
$G_0 \times_{I^2(\omega)} \aa^{a}$ is a local smooth $G_0$-set, and $(1,0)$ is a fixed point. 
\end{proposition}
\begin{proof} We define an open subset $\Omega$ and a smooth map $\rho$ by
\begin{enumerate}
\item $\Omega = \{(g,(h,b)) \in G_0 \times (G_0 \times _{I^2(\omega)} \aa^{a}) | (g,h) \in U_{2-loc}\}$.
\item $\rho(g,(h,b)) = (g \conj h, g.b + I^2(\omega)(g,h))$.
\end{enumerate} 
Let $(h,(k,z)), (gh,(k,z)), (g,\rho(h,(k,z))) \in \Omega$. This is equivalent to the condition $(h,k), (gh,k), (g,h \conj k) \in U_{2-loc}$, that is $(g,h,k) \in U_{3-loc}$. We have
\begin{align*}
\rho(g,\rho(h,(k,z))) &= (g \conj (h \conj k), g.(h.z) + g.I^2(\omega)(h,k) + I^{2}(\omega)(g, h \conj k)).
\end{align*}
Using the identities $\eqref{cocycle augmented rack identity}$	and $(gh) \conj k = g \conj (h \conj k)$, we have
\begin{align*}
\rho(g,\rho(h,(k,z))) &= ((gh) \conj k, (gh).z + I^2(\omega)(gh,k)) = \rho(gh,\rho(k,z)).  
\end{align*}
Moreover, $\rho(1,(k,z)) = (1 \conj k, 1.z + I^2(\omega)(1,k)) = (k,z)$ and $\rho(g,(1,0)) = (g \conj 1, g.0 + I^2(\omega)(g,1)) = (1,0)$. Hence $G_0 \times_{I^2(\omega)} \aa^{a}$ is a local smooth $G_0$-set and $(1,0)$ is a fixed point for this local action.
\end{proof}
We remark that we can reconstruct the rack product in $G_0 \times_{I^2(\omega)} \aa^{a}$ from the formula of the $G_0$-action. Indeed, we have $(g,a) \conj (h,b) = g.(h,b) = p(g,a).(h,b)$
where $p$ is the projection on the first factor $G_0 \times_{I^2(\omega)} \aa^{a} \fldrs{p} G_0$. Because $p(1,0) = 1$ and $p$ is equivariant we have shown the following proposition.
\begin{proposition} \label{G_0 x a$ est un rack augmenté}
$G_0 \times_{I^2(\omega)} \aa^{a} \fldrs{p} G_0$ is a local augmented Lie rack.
\end{proposition}
Hence we can rewrite our main theorem
\begin{theoreme}
Every Leibniz algebra $\gg$ can be integrated into a local augmented Lie rack of the form 
$$
G_0 \times_{I^2(\omega)} \aa^{a} \fldrs{p} G_0,
$$
with local action
$$
g.(h,b) = (g \conj h, g.b + I^{2}(\omega)(g,h)),
$$
and neutral element $(1,0)$, where $G_0$ is a Lie group, $\aa$ a representation of $G_0$ and $\omega \in ZL^2(\gg_0,\aa^{a})$. Conversely, every local augmented Lie rack of this form has its tangent space at $1$ provided with a Leibniz algebra structure.
\par
Moreover, in the special case where $\gg$ is a Lie algebra, the above construction provides $G_0 \times_{I^2(\omega)} \aa^a$ with a rack product coming from the conjugation in a Lie group. Conversely, if the tangent space at $1$ of $G_0 \times_{I^2(\omega)} \aa^{a}$ is a Lie algebra, then $G_0 \times_{I^2(\omega)} \aa^{a}$ can be provided with a Lie group structure, and the conjugation induced by the Lie group structure is the one defined by $\eqref{conjugation}$.
\end{theoreme}
\subsection{Example of a non-split Leibniz algebra integration}
In this section we construct the Lie rack associated to a Leibniz algebra of dimension 5 by following the method explained above. Other examples of integration in dimension 4 can be found in \cite{Covez_Thesis}. 
\vskip 0.2cm
Let $\gg = \R^5$. We define a bilinear map on $\gg$ by
\begin{align*}
[e_1,e_1] &= [e_1,e_2] = e_3 \\
[e_2,e_1] &= [e_2,e_2] = [e_1,e_3] = e_4 \\
[e_1,e_4] &= [e_2,e_3] = e_5
\end{align*}
We have $[(x_1,x_2,x_3,x_4),(y_1,y_2,y_3,y_4)] = (0,0,x_1(y_1 + y_2),x_2(y_1 + y_2) + x_1y_3,x_1y_4 + x_2y_3)$ and $(\mathfrak{g},[-,-])$ is a Leibniz algebra.
\par
%
%
To follow the integration method explained above, we have to determine the left center $Z_L(\gg)$, the quotient of $\gg$ by $Z_L(\gg)$ denoted $\gg_0$, the action of $\gg_0$ on $Z_L(\gg)$ and the Leibniz $2$-cocycle describing the abelian extension $Z_L(\gg) \fldri{} \gg \fldrs{} \gg_0$.
\vskip 0.2cm
Let $x \in Z_L(\gg)$, for $y=(0,0,1,0,0)$ in $\gg$, we have $[x,y] = 0$. This implies that $x_1 = x_2 = 0$. Conversely, every element in $\gg$ with the first two coordinates equal to $0$ is in $Z_L(\gg)$. Hence $Z_L(\gg) = < e_3,e_4,e_5 >$ and $\gg_0 \simeq < e_1,e_2>$. The bracket on $\gg_0$ is equal to zero, hence $\gg_0$ is an abelian Lie algebra. The action of $\gg_0$ on $ZL(\gg)$ is given by 
$$
\rho_x(y) = [(x_1,x_2,0,0,0),(0,0,y_3,y_4,y_5)] = (0,0,0,x_1y_3,x_1y_4+x_2y_3), 
$$
and the Leibniz $2$-cocycle is given by
$$
\omega(x,y) = [(x_1,x_2,0,0,0),(y_1,y_2,0,0,0)] = (0,0,x_1(y_1 + y_2),x_2(y_1 + y_2),0).
$$
\par
Moreover, we have $[x,x] = (0,0,x_1(x_1 + x_2),x_2(x_1+x_2)+x_1x_3,x_1x_4+x_2x_3)$, hence taking $x=(1,0,0,0,0),(0,1,0,0,0)$ and $(0,1,1,0,0)$, we see easily that $\gg_{ann} = Z_L(\gg)$. This Leibniz algebra is not split because for $\alpha \in Hom(\gg,Z_L(\gg))$ and $x,y \in \gg_0$, we have $d_L\alpha(x,y) = \rho_x(\alpha(y)) = (0,0,0,x_1\alpha(y)_3,x_1\alpha(y)_4+x_2\alpha(y)_3)$.\vskip 0.2cm
Now, we have to determine the Lie group $G_0$ associated to $\gg_0$, the action of $G_0$ on $Z_L(\gg)$ integrating $\rho:\gg_0 \to End(Z_L(\gg))$ (the action of $\gg_0$ on $Z_L(\gg)$), and the Lie rack cocycle integrating $\omega$.
\vskip 0.2cm
The Lie algebra $\gg_0$ is abelian, thus a Lie group integrating $\gg_0$ is $G_0 = \gg_0$. To integrate the action $\rho$, we use the exponential $\exp:End(Z_L(\gg)) \to Aut(Z_L(\gg))$. Indeed, for all $x \in \gg_0$, we have 
$$
\rho_x = \begin{pmatrix} 0 & 0 & 0 \\ x_1 & 0 & 0 \\ x_2 & x_1 & 0 \end{pmatrix}. 
$$
Hence, we define a Lie group morphism $\phi:G_0 \to Aut(Z_L(\gg))$ by setting
$$
\phi_x = \exp(\rho_x) = \begin{pmatrix} 1 & 0 & 0 \\ x_1 & 1 & 0 \\ x_2 + \frac{1}{2}x_1^2 & x_1 & 0 \end{pmatrix}.
$$
It is easy to see that $d_1\phi = \rho$. What remains to be done is the integration of the cocycle $\omega$. A formula for $f$, a Lie rack cocycle integrating $\omega$, is
$$
f(a,b) = \int_{\gamma_b}(\int_{\gamma_a} \tau^2(\omega)^{eq})^{eq},
$$
where $\gamma_{a}(s) = sa$ and $\gamma_b(t) = tb$. Let $a \in G_0$ and $x,y \in \gg_0$. We have
\begin{align*}
\int_{\gamma_a} \tau^2(\omega)^{eq} &= \int_{[0,1]} \tau^2(\omega)^{eq}(\gamma_a(s))(\dpart{s}\gamma_a(s)) ds \\
							    &= \int_{[0,1]} \phi_{\gamma_a(s)} \circ \tau^2(\omega)(a) ds 
\end{align*}
Thus $\displaystyle \int_{\gamma_a} \tau^2(\omega)^{eq} = \begin{pmatrix} a_1 & a_1 \\ \frac{1}{2}a_1^2 + a_2 & \frac{1}{2}a_1^2 + a_2 \\ a_1a_2 + \frac{1}{6} a_1^3 & a_1a_2 + \frac{1}{6} a_1^3 \end{pmatrix}$ and 
\begin{align*}
f(a,b) = \int_{\gamma_b} (\int_{\gamma_a} \tau^2(\omega)^{eq})^{eq}
	 = \int_{[0,1]} \phi_{\gamma_b(t)}(\begin{pmatrix} a_1 & a_1 \\ \frac{1}{2}a_1^2 + a_2 & \frac{1}{2}a_1^2 + a_2 \\ a_1a_2 + \frac{1}{6} a_1^3 & a_1a_2 + \frac{1}{6} a_1^3 \end{pmatrix}(b))  dt.
\end{align*}
We have 
\begin{align*}
\phi_{\gamma_b(t)}(\begin{pmatrix} a_1 & a_1 \\ \frac{1}{2}a_1^2 + a_2 & \frac{1}{2}a_1^2 + a_2 \\ a_1a_2 + \frac{1}{6} a_1^3 & a_1a_2 + \frac{1}{6} a_1^3 \end{pmatrix}(b)) &= \begin{pmatrix} 1 & 0 & 0 \\ tb_1 & 1 & 0 \\ tb_2 + \frac{1}{2}(tb_1)^2 & tb_1 & 0 \end{pmatrix}\begin{pmatrix} a_1 & a_1 \\ \frac{1}{2}a_1^2 + a_2 & \frac{1}{2}a_1^2 + a_2 \\ a_1a_2 + \frac{1}{6} a_1^3 & a_1a_2 + \frac{1}{6} a_1^3 \end{pmatrix}\begin{pmatrix} b_1 \\ b_2 \end{pmatrix} \\
															       &= \begin{pmatrix} a_1(b_1 + b_2) \\ (tb_1a_1 + a_2 + \frac{1}{2}a_1^2)(b_1 + b_2) \\ (a_1a_2 + \frac{1}{6}a_1^3 + \frac{1}{2}tb_1a_1^2 + tb_2a_1 + tb_1a_2 + \frac{1}{2}(tb_1)^2a_1)(b_1 + b_2) \end{pmatrix}. 
\end{align*}
Thus
$$
f(a,b) = \begin{pmatrix} a_1(b_1 + b_2) \\ (\frac{1}{2}b_1a_1 + a_2 + \frac{1}{2}a_1^2)(b_1 + b_2) \\ (a_1a_2 + \frac{1}{6}a_1^3 + \frac{1}{4}b_1a_1^2 + \frac{1}{2}b_2a_1 + \frac{1}{2}b_1a_2 + \frac{1}{6}(b_1)^2a_1)(b_1 + b_2) \end{pmatrix}. 
$$
and the conjugation in $G_0 \times_{f} Z_L(\gg) = \R^5$ is given by
\begin{align*}
\begin{pmatrix} a_1 \\ a_2 \\ a_3 \\ a_4 \\ a_5 \end{pmatrix} \conj \begin{pmatrix} b_1 \\ b_2 \\ b_3 \\ b_4 \\ b_5 \end{pmatrix} &= \begin{pmatrix} b_1 \\ b_2 \\ b_3 + a_1(b_1 + b_2) \\ a_1b_3 + b_4 + (\frac{1}{2}b_1a_1 + a_2 + \frac{1}{2}a_1^2)(b_1 + b_2) \\(a_2 + \frac{1}{2}a_1^2)b_3 + a_1b_4 + b_5 + (a_1a_2 + \frac{1}{6}a_1^3 + \frac{1}{4}b_1a_1^2 + \frac{1}{2}b_2a_1 + \frac{1}{2}b_1a_2 + \frac{1}{6}(b_1)^2a_1)(b_1 + b_2) \end{pmatrix}.
\end{align*}
%
\bibliographystyle{alpha}
\bibliography{bibliography}
\end{document}